\documentclass[12pt,reqno]{amsart}
\usepackage{geometry}
\usepackage{graphicx}				% Use pdf, png, jpg, or eps?¡ì with pdflatex; use eps in DVI mode
\usepackage{indentfirst}
\usepackage{hyperref}
\usepackage{ulem}	
\usepackage{amsthm}
\usepackage{amsfonts,bm}
\usepackage{amsmath}
\allowdisplaybreaks[4]
\usepackage{amscd}
\usepackage{enumerate,footnote}
\usepackage{amssymb}	
\usepackage{verbatim}
\usepackage{cancel}
\usepackage{cite}
\usepackage{textcomp}								
\usepackage{color}
\usepackage{mathrsfs}
\usepackage{appendix}

\newtheorem{theorem}{Theorem}[section]
\newtheorem{definition}{Definition}[section]
\newtheorem{proposition}{Proposition}[section]
\newtheorem{lemma}{Lemma}[section]
\newtheorem{remark}{Remark}[section]
\newtheorem{question}{Question}[section]
\newtheorem{corollary}{Corollary}[section]

\def\rd{\mathrm d}

\def\be{\begin{equation}}
\def\ee{\end{equation}}

\title[Slow Entropy For Abelian Actions]{Slow Entropy For Abelian Actions}
\author{Changguang Dong}
\address{Chern Institute of Mathematics and LPMC, Nankai University, Tianjin 300071, China}
\email{dongchg@nankai.edu.cn}

\author{Qiujie Qiao}
\address{Chern Institute of Mathematics and LPMC, Nankai University, Tianjin 300071, China}
\email{qiujieqiao@mail.nankai.edu.cn}
\date{\today}


\begin{document}

\begin{abstract}
We calculate slow entropy type invariant introduced by A. Katok and J.-P. Thouvenot in \cite{KT97} for higher rank smooth abelian actions for two leading cases: when the invariant measure is absolutely continuous and when it is hyperbolic. We generalize Brin-Katok local entropy Theorem to the abelian action for the above two cases. We also prove that, for abelian actions, the transversal Hausdorff dimensions are universal, i.e. dependent on the action but not on any individual element of the action.
\end{abstract}

\maketitle

\tableofcontents

\section{Introduction and Main Results}\label{set: introduction}

Metric entropy is an important numerical invariant in dynamical systems. 
It reflects exponential orbit growth rate of a system in measure theoretic sense, 
which is well studied in smooth ergodic theory for $\mathbb{Z}$- and $\mathbb{R}$- actions. 
However, if we consider higher rank abelian actions, and want to measure the complexity of 
such system, the direct extension of metric entropy fails to be useful. 
In most cases, it is equal to zero unless some or all transformations have infinite metric entropy, 
see \cite{KKH14,OW83,Lin01}. 
So, there is a need to find some other entropy type invariants.

One natural way is to change the normalization and measure exponential growth rate against 
the radius of the ball in the acting group instead of the volume of the ball. 
Very similar to Katok's definition in \cite{Ka80}, slow entropy type invariants for 
abelian actions have been defined in \cite{KT97}, 
and further studied by A. Katok, S. Katok and F. Rodriguez Hertz in \cite{KKH14}. 
In the latter paper, they consider the case of Cartan actions on the torus and 
find some connection with Fried average entropy (see \cite{KKH14} and the references therein).

There are various types of slow entropy in the literature, in contrast to the classical measure-theoretic entropy and topological entropy. Roughly speaking, slow entropy provides a much more precise measurement of complexity for
both homogeneous and non-homogeneous dynamical systems, 
with representative results including the following:
\begin{itemize}
\item Homogeneous systems: 
\begin{itemize}
\item[(1)] Fried average entropy and slow entropy for actions of higher-rank abelian groups \cite{KKH14};
\item[(2)] Parabolic flows and quasi-unipotent flows on homogeneous spaces \cite{KVW19};
\item[(3)] Abelian unipotent actions on finite-volume homogeneous spaces \cite{KKVWpp}.
\end{itemize}
\item Non-homogeneous systems: 
\begin{itemize}
\item[(1)] Certain classes of smooth mixing flows on surfaces \cite{K18};
\item[(2)] Genericity and rigidity of transformations characterized by slow entropy \cite{A21};
\item[(3)] Flexibility in the values of upper and lower polynomial slow entropy for rigid transformations \cite{BKWpp};
\item[(4)] Topological and measure-theoretic slow entropy of Anosov-Katok diffeomorphisms \cite{BKWpp-2}.
\end{itemize}
\end{itemize}
For a more comprehensive overview of the history, background
and further references on slow entropy, we refer the reader to the survey article \cite{KKWpp}. From now on, we will speak of the slow entropy for abelian actions defined in \cite{KT97} 
as simply the slow entropy.

In this paper, we consider this slow entropy for abelian actions of more general type. 
An explicit formula is given for that, which is our main result. 
Before that, let's make some basic settings throughout this paper. 

Let $(M,d)$ be a compact smooth manifold with a metric $d$, $m=$dim$M$, 
and $\alpha: \mathbb{R}^k\to \text{Diff}^{1+r}(M)  (r>0)$ be a locally free $\mathbb{R}^k$-action 
on $M$; $\mu$ is an invariant Borel probability measure for $\alpha$, 
and also assume it is ergodic; let $p$ be an arbitrary norm on $\mathbb{R}^k$. 
We say, an invariant measure $\mu$ is {\bf hyperbolic} if there exists $m-k$ nontrivial exponents, 
equivalently there exists a ${\bf t}\in \mathbb{R}^k$ such that $\alpha({\bf t})$ has $m-k$ nonzero exponents. 
Let $\{\chi_i\}_{1\leq i \leq L}$ be the Lyapunov exponents in Lyapunov decomposition, and let 
$\gamma_i({\bf t})$ be the corresponding transversal Hausdorff dimension (THD) for $\chi_i({\bf t})$ 
(see sections \ref{subset:Lyapunov Exponents, Suspension, Charts} and \ref{subset:Transversal Hausdorff Dimension(THD)} 
for the detailed definition). 
Note that, by definition, $\gamma_i({\bf t})$ is defined to be $\gamma_i(-{\bf t})$ when $\chi_i({\bf t})<0$. 
Hence the domain of $\gamma_i({\bf t})$ is $\{{\bf t}\in\mathbb{R}^k: \chi_i({\bf t})\neq 0\}$. 
As a preparation for the slow entropy formula, we first give the following general result on 
the universality of THDs, which can be used independently. Let's mention here that, 
it is known by certain amount of dynamists, however there is no proof yet.

\begin{theorem}\label{THD}
As a function of ${\bf t}$, $\gamma_i({\bf t})$ is a nonnegative constant in $\{{\bf t}: \chi_i({\bf t})\neq 0\}$. Moreover, if we do not assume $\mu$ to be ergodic, then $\gamma_i({\bf t})$ is a nonnegative constant in each ergodic component of $\mu$.
\end{theorem}

Now we are ready to state our main result:

\begin{theorem}[{{\bf{Slow Entropy Formula}}}]\label{mainresult}
For abelian action $\alpha$, assume $\mu$ is either hyperbolic or 
absolutely continuous with respect to a volume form on $M$, then 
\begin{equation}
sh_\mu(\alpha,p)=\sum_{i=1}^{L}\gamma_i \max_{{\bf t}: p({\bf t})\leq 1} \chi_i({\bf t}).
\end{equation}
\end{theorem}

For the detailed definition of slow entropy, see section \ref{subset:Slow Entropy Type Invariants}. Here, one can easily see that, 
slow entropy is always finite if every element has finite metric entropy; 
and it does not vanish unless every element has zero metric entropy. 
Careful reader may find the similarity between the above formula and Ledrappier-Young formula 
for metric entropy (Theorem $C'$) in \cite{LY-85-annals-2}, 
and when $k=1$, $p$ is the standard norm, it reduces to the usual metric entropy case. 
So here we will call it slow entropy type Ledrappier-Young formula, 
though we can only prove it under some restrictions on the measure. 

Here, it is also important to note that, measure rigidity results for higher rank abelian group 
actions, especially those from \cite{KKH11} and \cite{KH16}, 
indicate that the case of absolutely continuous measure is indeed the central one. 
In this case, $\gamma_i$ will be the multiplicity of the corresponding exponent $\chi_i$, 
and the formula then becomes the slow entropy version of Pesin entropy formula.

Here, for $n\in\mathbb{N}$ and $\epsilon>0$, we define the Bowen ball $ B(\alpha, F^p_n, x, \epsilon)$ by 
$$
B(\alpha, F^p_n, x, \epsilon)=\{y\in M: d(\alpha({\bf t})x,\alpha({\bf t})y)\leq \epsilon, \,\,\forall\, {\bf t}\,\, s.t.\,\, p({\bf t})\leq n\}.
$$ 
As a by-product, we also prove the following generalized Brin-Katok local entropy Theorem:

\begin{theorem}\label{bk}
Under the same assumptions as in Theorem \ref{mainresult}, for $\mu$ a.e. $x$, 
\begin{align*}
\lim_{\epsilon\to 0}\liminf_{n\to \infty}\frac{-\log \mu(B(\alpha, F^p_n, x, \epsilon))}{n}
=\lim_{\epsilon\to 0}\limsup_{n\to \infty}\frac{-\log \mu(B(\alpha, F^p_n, x, \epsilon))}{n},
\end{align*}
and this limit is equal to 
$$
\sum_{i=1}^{L}\gamma_i \max_{{\bf t}: p({\bf t})\leq 1} \chi_i({\bf t}).
$$
\end{theorem}

In fact, most of our work goes into proving this theorem, 
and then Theorem \ref{mainresult} is an easy consequence.
 
Let us point out the main difficulties in proving Theorem \ref{mainresult}. 

Recall that, for metric entropy of diffeomorphisms, 
Brin-Katok Theorem on local entropy \cite{BK83}, 
Shannon-McMillan-Breiman(SMB) Theorem and partition theory 
(Sinai partition \cite{Sin68}) are highly used, 
see \cite{LY-85-annals-1,LY-85-annals-2}. 
However, for abelian actions, SMB Theorem is not that useful, 
because the extension of SMB Theorem for actions \cite{OW83, Lin01} 
includes faster growth of the denominator than what is needed in our case. 

Another difficulty is that, we heavily use a local entropy type theorem (Theorem \ref{bk}) 
to prove slow entropy formula, but we can not prove it in the general case, 
because we can neither generalize the proof of Brin-Katok Theorem to our case 
(which highly uses SMB Theorem) nor come out with a new proof. 
As a result, we have to put extra assumptions on the measure into our main result. 
In addition, unfortunately there is no way to construct an increasing partition 
for the action, hence we lose many powerful tools from partition theory.

In contrast to the metric entropy, another huge problem 
we can not avoid is the existence of zero Lyapunov exponents, which 
equivalently speaking, the case of non-hyperbolic measure for actions. 
Hyperbolic measure of a $C^{1+r}$ diffeomorphism locally has so called asymptotically 
almost local product structure. Namely, such kind of measure is exact dimensional, 
see \cite{BPS99} for details. 
The proof in \cite{BPS99} essentially exploits results from \cite{LY-85-annals-1,LY-85-annals-2}, 
and uses a combinatorial argument based on a special partition constructed in \cite{LY-85-annals-1,LY-85-annals-2}. 
If we just consider hyperbolic measure for abelian actions, then similar method allows us to handle the problem. However, due to the existence of zero Lyapunov exponents, it is difficult to control the behavior in the neutral directions. This is a very subtle issue in dimension theory and smooth ergodic theory.

A similar problem is to give a close enough lower bound of the lower pointwise dimension 
not only for hyperbolic measure but for arbitrary Borel probability invariant measures, 
which should be similar to Theorem F for upper pointwise dimension in \cite{LY-85-annals-2}. 
For example, in \cite{LY-85-annals-2} the following quantities (whenever they are well defined) 
are considered, which are called stable and unstable pointwise dimensions of measure $\mu$, 
\begin{align*}
d^s(x)&:=\lim_{r\to 0}\frac{\log{\mu_x^s(B^s(x,r))}}{\log{r}};\\
d^u(x)&:=\lim_{r\to 0}\frac{\log{\mu_x^u(B^u(x,r))}}{\log{r}}; 
\end{align*} 
here see \cite{LY-85-annals-2} or \cite{BPS99} for more details. 
Recently in \cite{ORH23}, 
given a $C^{1+\alpha}$ diffeomorphism $f$ preserving an invariant ergodic measure $\mu$, 
Ben Ovadia and Rodriguez Hertz proved that for $\mu$-almost every $x$,
$$
d^s(x)+d^u(x)\leq \uline{d}(x):=\liminf_{r\to 0}\frac{\log{\mu(B(x,r))}}{\log{r}}.
$$

Finally, let us emphasize here, slow entropy type invariant may have some applications to the study of Kakutani (or orbit) equivalence and rigidity problems of actions of higher rank abelian groups, which is our subsequent study in the future. 
 
In this paper, we will heavily use results and methods from \cite{LY-85-annals-1,LY-85-annals-2}. 
We also use an important technique from H. Hu's paper \cite{Hu93} to prove Theorem \ref{THD}. 
For the proof of Theorem \ref{bk}, 
we apply a combinatorial argument from \cite{LY-85-annals-1,LY-85-annals-2} 
and a sub-exponential measure density lemma for Bowen balls. 
The specific form of the density lemma is due to Ben Ovadia and Rodriguez Hertz \cite{ORH23}.

\subsection*{Outline of the paper}
This paper is organized as follows. 
In \S \ref{set:Preliminaries}, we present some definitions and settings.
The proof of Theorem \ref{THD} is given in \S \ref{set:Transversal Hausdorff Dimensions}. 
The principal and essential parts of this work are \S \ref{set:Main Reduction}, \S\ref{set:Hyperbolic Case} and 
\S\ref{set:Absolutely Continuous Case}, where we prove Theorem \ref{bk}, as well as Theorem \ref{mainresult}. 
In the last section, we address some open questions and possible characterization of slow entropy.
In the Appendix, we provide the equivalent definitions of slow entropy for smooth abelian actions.

\subsection*{Acknowledgement} 

C. Dong would like to express his great gratitude to  Anatole Katok for giving him the problem more than 10 years ago, as well as subsequent discussions. This project, especially the content in the appendix would never be finished without his involvement. He also thank Yakov Pesin, Federico Rodriguez Hertz and Omri Sarig for discussions on related topics. 
C. Dong is grateful to Kurt Vinhage and Shilpak Banerjee for helps during the preparation of the very first draft.

This research was partially supported by National Key R\&D Program of China No. 2024YFA1015100. C. Dong was also grateful for the support by Nankai Zhide Foundation, ``the Fundamental Research Funds for the Central Universities" No. 100-63243066 and 100-63253093.
Q. Qiao was supported by Nankai Zhide Foundation.

\section{Preliminaries}\label{set:Preliminaries}

\subsection{Lyapunov Exponents, Suspension, Charts}\label{subset:Lyapunov Exponents, Suspension, Charts}

Let $T_x M$ be the tangent space of $M$ at $x$, and for ${\bf t}\in \mathbb{R}^k$, $\alpha({\bf t})$ induces a map $D_x\alpha({\bf t}): T_xM\to T_{\alpha({\bf t})x}M$. One may always assume that $k\ge2$, otherwise it will reduce to the usual case (flow). For simplicity, we will use ${\bf t}$ as the diffeomorphism instead of $\alpha({\bf t})$ in some cases.

Let's first consider a $\mathbb{Z}^k$ action. According to the Multiplicative Ergodic Theorem, 
there exists a measurable set $\Gamma$ with $\mu(\Gamma)=1$, such that for all $x\in\Gamma$, 
nonzero $u\in T_xM$, such that for every ${\bf t}\in \mathbb{Z}^k$ 
the limit$$\chi(x,u, \alpha({\bf t}))=\lim_{n\to \infty}\frac{\log\|D_x\alpha(n{\bf t}) u\|}{n}$$ exists 
and we call it the Lyapunov exponent of $u$ at $x$ for $\alpha({\bf t})$. One can easily see that, 
for each $t$, the Lyapunov exponent can only take finite numbers. 
Since $\alpha$ is an abelian action, we can get a common splitting for the tangent space 
$TM=\bigoplus E_{\chi}$. And also, since $\mu$ is ergodic, $\chi$ is independent on $x$. 
Thus we will only denote $\chi_i({\bf t})$ to be the $i$-th Lyapunov exponent for $\alpha({\bf t})$. 
And the common refinement $TM=\bigoplus E_{\chi_i}$ is called the 
\textbf{Lyapunov decomposition} for $\alpha$.

For each $\chi_i$, viewed as a function of ${\bf t}$, is a linear functional from $\mathbb{Z}^k$ 
to $\mathbb{R}$. It can be linearly extended to a functional on $\mathbb{R}^k$. 
The hyperplanes $\ker\chi_i\subset \mathbb{R}^k$ are called the \textbf{Lyapunov hyperplanes} 
and the connected components of $\mathbb{R}^k\setminus\bigcup_i \ker\chi_i$ are called the 
\textbf{Weyl chambers} of $\alpha$. The elements in the union of the Lyapunov hyperplanes are 
called \textbf{singular}, and elements in the union of Weyl Chambers are called \textbf{regular}.
 For more details on the general theory, see \cite{KN11}.

Now given a $\mathbb{Z}^k$ action on $M$, let $\mathbb{Z}^k$ act on $\mathbb{R}^k\times M$ 
by $${\bf t} ({\bf s}, m)= ({\bf s}-{\bf t}, {\bf t} m)$$ and form the quotient space 
$$
S=\mathbb{R}^k \times M/\mathbb{Z}^k\cong \mathbb{T}^k \times M.
$$ 
Note that the action of $\mathbb{R}^k$ on $\mathbb{R}^k\times M$ by ${\bf s}({\bf t},m)=({\bf s}+{\bf t},m)$ 
commutes with the $\mathbb{Z}^k$ action and therefore we can get a  $\mathbb{R}^k$ action on $S$. 
This action is closely related to the original action, 
and we call it the \textbf{suspension} of $\mathbb{Z}^k$ action. 
In fact, when $k=1$, it is the usual suspension for one diffeomorphism. 
We can build a natural correspondence between invariant measures, 
nonzero Lyapunov exponents and stable/unstable distributions etc. 
between the suspension and original $\mathbb{Z}^k$ action. For example, 
if the $\mathbb{Z}^k$ action preserves $\mu$, then $\mathbb{R}^k$ action preserves 
$\lambda\times\mu$, here $\lambda$ is the Lebesgue measure on $\mathbb{T}^k$. 
And this is why we mostly only need to deal with $\mathbb{R}^k$ actions in this paper.

As usual, $d$ denotes the Riemannian metric on $M$. 
We write 
\[
\mathbb{R}^{\dim M} = \mathbb{R}^{\dim E_1} \times \cdots \times \mathbb{R}^{\dim E_r},
\]
and for $x \in \mathbb{R}^{\dim M}$, let $(x_1, \dots, x_r)$ be its coordinates with respect to this splitting. 
Define
\[
|x| = \max_i |x_i|_i,
\]
where $|\cdot|_i$ is the Euclidean norm on $\mathbb{R}^{\dim E_i}$.  

Let
\[
R^i(\rho) = \{\, x_i \in \mathbb{R}^{\dim E_i} : |x_i| \leq \rho \,\}
\]
and
\[
R(\rho) = \{\, x \in \mathbb{R}^{\dim M} : |x| \leq \rho \,\}.
\]

The following is a result directly quoted from \cite{KKH11}.

\begin{proposition}[Proposition 2.1.\cite{KKH11}]\label{chart1}
Let $\alpha$ be a locally free $C^{1+r}$, action of $\mathbb{R}^k$ on a manifold $M$ preserving an ergodic invariant measure $\mu$. There are linear functionals $\chi_i$, $i=1,\cdots, L$, on $\mathbb{R}^k$ and an $\alpha$-invariant measurable splitting called the Lyapunov decomposition, of the tangent bundle of $M$ 
$$
TM=T\mathcal{O}\oplus\bigoplus^L_{i=1}E_i
$$ 
over a set of full measure $\widetilde{\Gamma}$, where $T\mathcal{O}$ is the distribution tangent to the $\mathbb{R}^k$ orbits, such that for any ${\bf t}\in \mathbb{R}^k$ and any nonzero vector $v\in E_i$ the Lyapunov exponent of $v$ is equal to $\chi_i({\bf t})$, i.e. 
$$
\lim_{n\to \pm \infty}\frac{\log{\|D_x(\alpha(n{\bf t}))v\|}}{n}=\chi_i({\bf t}),
$$ 
where $\|\cdot\|$ is any continuous norm on $TM$. Any point $x\in \Gamma$ is called a regular point.

Furthermore, for any $\epsilon>0$ there exist positive measurable functions $C_{\epsilon}(x)$ and $K_{\epsilon}(x)$ such that for all $x\in \Gamma$, $v\in E_i(x)$, ${\bf t}\in \mathbb{R}^k$, and $i=1, \cdots, L$, 
\begin{itemize}
\item[(1)] $C_{\epsilon}^{-1}(x)e^{\chi_i({\bf t})-\frac{1}{2}\epsilon \|\alpha({\bf t})\|}\|v\|\leq \|D_x(\alpha({\bf t}))v\|\leq C_{\epsilon}(x)e^{\chi_i({\bf t})+\frac{1}{2}\epsilon \|\alpha({\bf t})\|}\|v\|$;
\item[(2)] Angles $\angle(E_i(x), T\mathcal{O})\ge K_{\epsilon}(x)$ and $\angle(E_i(x), E_j(x))\ge K_{\epsilon}(x), i\neq j$;
\item[(3)] $C_{\epsilon}(\alpha({\bf t})x)\leq C_{\epsilon}(x)e^{\epsilon\|\alpha({\bf t})\|}$ and $K_{\epsilon}(\alpha({\bf t})x)\ge K_{\epsilon}(x)e^{-\epsilon\|\alpha({\bf t})\|}$.
\end{itemize}
\end{proposition}

Finally, let's now construct Lyapunov charts for the action $\alpha$. 
The following are a generalized proposition from \cite{Hu93} with some modification of notations. 
We include here for further use, and for simplicity omit the proof because it is similar to
 Proposition 4.1. in \cite{Hu93}.

We use $m$ to denote the dimension of the manifold $M$. 
Let $\|\cdot \|$ be the standard norm on $\mathbb{R}^k$, 
and let $|\cdot |$ be the usual norm on $\mathbb{R}^m$. 
Also let $B(\rho)$ $(\rho >0)$ be the ball in $\mathbb{R}^m$ centered at the origin with radius 
$\rho$. We also assume the action is ergodic.

Denote $\{{\bf t}_1, \cdots, {\bf t}_k\}$ as the standard basis for $\mathbb{Z}^k$ 
w.r.t. the norm $\|\cdot \|$ on $\mathbb{R}^k$, i.e. 
it will span $\mathbb{Z}^k$ via coefficients in $\mathbb{Z}$. 
For ${\bf t}_1$, we denote its exponents correspondingly as $\chi_1({\bf t}_1)>\cdots>\chi_{m({\bf t}_1)}({\bf t}_1)$. 
We define 
$$
\chi_+({\bf t}_1)=\min\{\chi_i({\bf t}_1), \chi_i({\bf t}_1)>0\},\quad 
\chi_-({\bf t}_1)=\max\{\chi_i({\bf t}_1), \chi_i({\bf t}_1)<0\}
$$ 
and 
$$
\Delta({\bf t}_1)=\min\{\chi_i({\bf t}_1)-\chi_{i+1}({\bf t}_1), i=1, \cdots, m({\bf t}_1)-1\}.
$$
Define $\chi_{\pm}({\bf t}_i)$, $\Delta(x, {\bf t}_i)$ similarly. 
Let $\epsilon>0$ be such that  
$$
0<\epsilon\leq \epsilon_0:=\frac{1}{200m k}\min\{\Delta({\bf t}_i), \chi_{\pm}({\bf t}_i),\; i=1,\cdots, k\}.
$$

\begin{proposition}\label{chart2}
For the $\epsilon$ defined above, there exists a measurable function $l:\widetilde{\Gamma}\to [0,\infty)$ with $l(\alpha({\bf t})x)\leq l(x)e^{\epsilon \|{\bf t}\|}$, and a set of embeddings $\Phi_x:B(l(x)^{-1})\to M$ at each point $x\in \Gamma$ such that the following holds:
\begin{itemize}
\item[(i)] $\Phi_x(0)=x$, and the preimages $D\Phi_x(0)^{-1}(T\mathcal{O}(x))$ and $R_i(x)=D\Phi_x(0)^{-1}(E_i(x))$ of $E_i(x)$ are mutually orthogonal in $\mathbb{R}^m$, where $E_i(x)$ is the Lyapunov subspace for some exponent in Lyapunov decomposition.
\item[(ii)] Let $\tilde{{\bf t}}_x=\Phi_{{\bf t}x}^{-1}\circ t\circ \Phi_x$ be the connecting map between the chart at $x$ and the chart at ${\bf t}x$. Then $(\widetilde{{\bf t}+{\bf s}})_x=\tilde{{\bf t}}_{({\bf s}x)}\tilde{{\bf s}}_x=\tilde{{\bf s}}_{({\bf t}x)}\tilde{{\bf t}}_x$ and $(\widetilde{{\bf t}+{\bf s}})_x^{-1}=\tilde{{\bf t}}_{({\bf s}^{-1}x)}^{-1}\tilde{{\bf s}}_x^{-1}=\tilde{{\bf s}}^{-1}_{({\bf t}^{-1}x)}\tilde{{\bf t}}_x^{-1}$ for any ${\bf s},{\bf t}$.
\item[(iii)] For any $q$, $1\leq q\leq m$, a nonzero vector $u\in E_q(x)$, $v\in T\mathcal{O}(x)$, 
$$
|u|e^{\chi_q(x, {\bf t})-\epsilon \|{\bf t}\|}\leq |D_x\tilde{{\bf t}}(0)u|\leq |u|e^{\chi_q(x, {\bf t})+\epsilon \|{\bf t}\|}, i=1,\cdots, k,
$$ 
$$
|v|e^{-\epsilon \|{\bf t}\|}\leq |D_x\tilde{{\bf t}}(0)v|\leq 2|v|,
$$
 $$|u|e^{\chi_q(x, {\bf t})+\chi_q(x, {\bf s})-\epsilon (\|{\bf t}\|+\|{\bf s}\|)}\leq |D_x\widetilde{({\bf t}+{\bf s})}(0)u|\leq |u|e^{\chi_q(x, {\bf t})+\chi_q(x, {\bf s})+\epsilon(\|{\bf t}\|+\|{\bf s}\|)},\; \forall \,{\bf t},{\bf s}.$$
\item[(iv)] Let $L(\Psi)$ be the Lipschitz constant of the function $\Psi$. Then for any ${\bf t}$ with $p({\bf t})\leq 1$, 
$$
L(\tilde{{\bf t}}_x-D\tilde{{\bf t}}_x(0))\leq \epsilon , \quad L(D\tilde{{\bf t}}_x)\leq l(x).
$$
\item[(v)] There exists a number $\hat{\varepsilon}>0$ depending on $\epsilon$ and the exponents such that $\forall x\in \Gamma,$ 
$$
|\tilde{{\bf t}}_x u|\leq e^{\hat{\varepsilon}\|{\bf t}\|}|u|,\quad \forall u\in B(e^{-\hat{\varepsilon}-\epsilon}l(x)^{-1}).
$$
\item[(vi)] For all $u, v\in B(l(x)^{-1})$, we have $$K^{-1}d(\Phi_xu, \Phi_x v)\leq |u-v|\leq l(x)d(\Phi_x u, \Phi_x v),$$ for some universal constant $K$.
\end{itemize}
\end{proposition}

We'll call such local charts $\{\Phi_x: x\in \widetilde{\Gamma}\}$ the $(\epsilon, l)$-charts. 
Let $\epsilon$ in Propositions \ref{chart1} and \ref{chart2} be the same, 
and 
$$
\Gamma_{l,\epsilon}:=\left\{x\in \widetilde{\Gamma}: l(x)\leq l,\,C_{\epsilon}(x)\leq l,\, K_{\epsilon}(x)\geq \frac{1}{l}\right\}.
$$ 
When $l$ is large enough and $\epsilon$ is small enough, we have $\mu(\Gamma_{l,\epsilon})>0$. Furthermore, we obtain $\displaystyle \mu\left(\bigcup_{l>0,\epsilon>0} \Gamma_{l,\epsilon}\right)=1$.

\subsection{Transversal Hausdorff Dimension(THD)}\label{subset:Transversal Hausdorff Dimension(THD)}

Now we just consider one $C^{1+r}$ diffeomorphism $f:=\alpha({\bf t})$ on $M$ 
for some ${\bf t}\in\mathbb{R}^k$. It is a well-known fact that we can choose $t$ properly 
such that $f$ is ergodic with respect to $\mu$. The following are some definitions 
and results from section 7 in \cite{LY-85-annals-2}.

For a diffeomorphism $f$ with an ergodic measure $\mu$, let 
$$
\chi_1>\chi_2>\cdots>\chi_L
$$ 
denote its distinct Lyapunov exponents, and let 
$$
TM=E_1\oplus \cdots \oplus E_L
$$ 
be the corresponding Oseledec decomposition of its tangent space. 
Note that these exponents are all well-defined $\mu$ a.e. Let $u=\max\{i:\chi_i>0\}$, 
and for $1\le i\le u$, we define the $i$-th unstable manifold $W^i(x)$ of $f$ at $x$ by 
$$
W^i(x)=\bigg\{y\in M:\limsup_{n\to \infty} \frac{\log{d(f^{-n}x,f^{-n}y)}}{n}\le -\chi_i \bigg\},
$$ 
here $d(\cdot , \cdot )$ is a Reimannian metric on $M$. 
 Clearly, these form a nested family of foliations 
$$
W^1\subset W^2\subset \cdots \subset W^u.
$$

Each $W^i(x)$ inherits a Riemannian structure from $M$, and hence gives a metric 
on each leaf of $W^i$, which is denoted by $d^i$. 
We use $B^i(x,\epsilon)$ to denote the $d^i$-ball in $W^i(x)$ centered at $x$ of radius $\epsilon$. 

Given a measurable partition $\xi$ of $M$ 
which is subordinate to the $W^i$-foliation, there is a system of conditional measures 
induced from $\mu$ associated to each atom of $\xi$. 
In fact, these measures are defined up to a scalar multiple. 
We denote by $\{\mu_x^i\}$ the conditional measures for $\xi$.
For $x\in \Gamma$, we define the upper Hausdorff dimension and lower Hausdorff dimension of $W^i$ by 
$$
\bar\delta_i(x,\xi)=\limsup_{\epsilon\to 0}\frac{\log{\mu_x^i(B^i(x,\epsilon))}}{\log{\epsilon}}
$$ 
and 
$$
\uline\delta_i(x,\xi)=\liminf_{\epsilon\to 0}\frac{\log{\mu_x^i(B^i(x,\epsilon))}}{\log{\epsilon}}.
$$
 In \cite{LY-85-annals-2}, Ledrappier and Young proved 
 $\bar\delta_i(x,\xi)=\uline\delta_i(x,\xi)$ for $\mu$-a.e. $x\in M$, 
 and this value is independent of the partition $\xi$. 
 By ergodicity, this value is constant $\mu$-almost everywhere. 
 We denote this constant by $\delta_i$, 
 and refer to it as the dimension of $\mu$ along the foliation $W^i$.

For $\epsilon>0$, $x\in \Gamma$, $1\leq i\leq u$ and $n\in \mathbb{N}$, we define 
$$
V^i(x,n,\epsilon)=\big\{y\in W^i(x): d^i(f^kx,f^ky)<\epsilon\;\text{for}\; 0\leq k<n\big\}.
$$ 
Then we define 
$$
\bar{h}_i(x,\epsilon,\xi)=\limsup_{n\to\infty}-\frac{\log{\mu_x^i(V^i(x,n,\epsilon))}}{n} 
$$
and
$$
\uline{h}_i(x,\epsilon,\xi)=\liminf_{n\to\infty}-\frac{\log{\mu_x^i(V^i(x,n,\epsilon))}}{n}.
$$ 
According to Proposition 7.2.1 in \cite{LY-85-annals-2}, 
we conclude that for $\mu$-almost every $x\in M$, 
$$
\lim_{\epsilon\to 0}\uline{h}_i(x,\epsilon,\xi)=\lim_{\epsilon\to 0}\bar{h}_i(x,\epsilon,\xi)\;\;\mu\; a.e. \,x,
$$ 
and this limit is independent of the choice of $\xi$ or $\{\mu_x^i\}$. 
As a function, this quantity is measurable, and hence by the ergodicity of $\mu$, 
it is constant $\mu$-almost everywhere. We denote this constant by $h_i$, 
whhich is called the entropy along the $i$-th unstable manifold.

A celebrated result (Theorem $C'$ in \cite{LY-85-annals-2}) states the following: 
\begin{itemize}
\item[(i)]  $h_1=\chi_1\delta_1$, 
\item[(ii)] $h_i-h_{i-1}=\chi_i(\delta_i-\delta_{i-1})$\;\text{for}\; $2\leq i\leq u$, 
\item[(iii)]  $h_u=h_{\mu}(f)$.
\end{itemize}
Let $\gamma_1=\delta_1$ and $\gamma_i=\delta_i-\delta_{i-1}$ for $i=2,\cdots, u$. 
By replacing $f$ with $f^{-1}$, we can define the quantities $\delta_{q+1-s}, \cdots, \delta_q$, 
where $s=\#\{i: \chi_i<0\}$ denotes the number of distinct negative Lyapunov exponents of $f$. 
For these indices, we set $\gamma_q=\delta_q$ 
and $\gamma_i=\delta_i-\delta_{i+1}$ for $i= q+1-s, \cdots, q-1$. 
For indices $i$ with $\chi_i=0$, we simply define $\gamma_i=\text{dim} E_{\chi_i}$. 
From the above definitions, it follows that 
$$
\sum_i \gamma_i \chi_i=0,
$$ 
and 
$$
\sum_i\gamma_i\left\vert \chi_i\right\vert=2h_{\mu}(f).
$$

Here, $\gamma_i$ is called the {\bf transversal Hausdorff dimension (THD)} of $\mu$ 
with respect to $\chi_i$. 
These quantities depend on the diffeomorphism $f$ and the measure $\mu$. 
However, in the abelian action case, as we will prove in next section, 
they do not depend on the choice of an element of the action.

The essential fact behind the above definitions and results is that all intermediate stable 
and unstable distribution are integrable. 
Specifically, $\displaystyle \bigoplus_{1\leq i\leq j} E^i$ are integrable for $1 \leq j\leq u$. 
One should be careful extending the definitions to abelian actions, 
since the $\gamma_i$ corresponding to $\alpha({\bf t})$ may sometimes
split into two or more THDs for some other $\alpha({\bf s})$.

\subsection{Slow Entropy Type Invariants}\label{subset:Slow Entropy Type Invariants}

There are two approaches to slow entropy for $\mathbb{Z}^k$ action. 

The first one is based on an idea of coding. 
Let \(\Gamma\) be a discrete group, \(F \subset \Gamma\) its subset. We consider the spaces
\[
\Omega_{N,F} = \{\omega = (\omega_\gamma)_{\gamma \in F}; \, \omega_\gamma \in \{1, \ldots, N\}\}
\]
For every \(F \subset \Gamma\), we define the projection by \(\pi_{\Gamma,F}: \Omega_{N,\Gamma} \to \Omega_{N,F}\).
For any finite set \(F \subset \Gamma\), we define the Hamming metric \(d_F^H\) in \(\Omega_{N,F}\) by
\[
d_F^H(\omega, \overline{\omega}) := \frac{1}{\operatorname{card}F} \sum_{\gamma \in F} (1 - \delta_{\omega_\gamma \overline{\omega}_\gamma}),
\]
where $\delta_{kl}$ is a Kronecker symbol:
\begin{align*}
\delta_{kl} =
\begin{cases}
0 & \text{if } k \neq l, \\
1 & \text{if } k = l.
\end{cases}
\end{align*}

Let \(T: (X,\mu) \times \Gamma \to (X,\mu)\) be an action of the group \(\Gamma\) by measure-preserving transformations of a Lebesgue space; let \(\xi = (c_1, \ldots, c_N)\) be a finite measurable partition. We define the ``coding map'' \(\phi_{T,\xi}: X \to \Omega_{N,\Gamma}\) by \((\phi_{T,\xi})_\gamma = \omega_\gamma(x)\) where \(T(\gamma)x \in c_{\omega_\gamma(x)}\). Partial coding \(\phi_{T,\xi}^F\) for \(F \subset \Gamma\) is defined by \(\phi_{T,\xi}^F = \pi_{\Gamma,F} \circ \phi_{T,\xi}\).
The partial coding \(\phi_{T,\xi}^F\) defines the measure \((\phi_{T,\xi}^F)_* \mu\) in \(\Omega_{N,F}\).

For $\epsilon>0$ and $\delta>0$, we use $S^H_\xi(T,F,\epsilon,\delta)$ to denote 
the minimal number of balls of radius $\epsilon$ (for the metric $d^H_F$ on $\Omega_{F,N}$) whose union has $\left(\phi^F_{T,\xi}\right)_*\mu$-measure $\ge 1-\delta$. See section 1.1 in \cite{KT97} for details. 

Given a norm $p$ on $\mathbb{R}^k$, let $F^p_s$ be the set of points in $\mathbb{Z}^k$ which are also contained in the ball centered at $0$ with radius $s$. We define the slow entropy of a $\mathbb{Z}^k$ action $\alpha$ with respect to the norm $p$ and the partition $\xi$ as 
$$
sh^{H}_{\mu}(\alpha, p, \xi)=\lim_{\epsilon, \delta \to 0}\limsup_{s\to\infty}\frac{\log{S^H_{\xi} (\alpha, F^p_s, \epsilon, \delta)}}{s}.
$$ 
Then we define 
$$
sh^{H}_{\mu}(\alpha, p)=\sup_{\xi} sh_{\mu}(\alpha,p,\xi), 
$$
which we call \textbf{Hamming entropy} for short.

 The other approach is to start with a metric $d$ on $M$, 
 and define $d_F=\max_{{\bf t}\in F} d\circ \alpha({\bf t})$. 
 Denote $S_d(\alpha, F, \epsilon,\delta)$ as the minimal number of $\epsilon$-$d_F$ 
 balls whose union has measure $\ge 1-\delta$. 
 For the same $F^p_s$, we define the \textbf{Bowen entropy} by 
 $$
 sh_{\mu}(\alpha, p)=\lim_{\epsilon, \delta \to 0}\limsup_{s\to\infty}\frac{\log{S_d (\alpha, F^p_s, \epsilon, \delta)}}{s}.
 $$

In the appendix, we prove that 
these two definitions coincide for smooth abelian actions, 
and hence for the latter it does not depend on the choice of $d$. 
Finally, slow entropy for $\alpha$ is defined as 
$$
sh_{\mu}(\alpha)=\inf_{p:vol(p)=1} sh_{\mu}(\alpha, p),
$$ 
here $vol(p)$ is the volume of the unit ball in the norm $p$. 

For any $x\in M$ and $\epsilon>0$.  
given $s>1$, there exists $n\in\mathbb{N}$ such that 
$$
B(\alpha, F^p_{n+1},x,\epsilon)\subset B(\alpha, F^p_s,x,\epsilon)\subset B(\alpha, F^p_n,x,\epsilon). 
$$
This implies that in the definition of the Bowen entropy $sh_\mu(\alpha,p)$, 
we can replace the Bowen ball $B(\alpha, F^p_s,x,\epsilon)$ 
with $B(\alpha, F^p_n,x,\epsilon)$ and replace the limit $s\to\infty$ (where $s\in \mathbb{R}$) with the limit $n\to\infty$ (where $n\in \mathbb{N}$). 

In this paper, we will mostly focus on the quantity $sh_{\mu}(\alpha, p)$ 
instead of $sh_{\mu}(\alpha)$. 
For $\mathbb{R}^k$ action, 
we use the definition corresponding to $sh_\mu(\alpha)$ to refer to slow entropy. 

In the case of a non-ergodic invariant measure, 
we follow the standard procedure convention: decompose the measure into its ergodic components, 
then integrate the slow entropy over all ergodic components. 
For arbitrary actions, this convention cannot be applied. However, it is valid in the smooth case.
Let us emphasize the role of the norm $p$: 
it can be interpreted as a time change, 
namely, changing the norm $p$ means a time change of the abelian action.

For more details and discussions about slow entropy, 
we refer the reader to Section 1 in \cite{KT97} and Section 3 in \cite{KKH14}.

A natural question arising here is whether the Hamming entropy coincides with the Bowen entropy. 
In the rank one case, this is true in general, i.e. 
or homeomorphisms and continuous flows on compact metric spaces, see \cite[Theorem 1.1]{Ka80}. 
In the higher rank situation, it remains open in the general case. While for smooth abelian actions, we can answer it in the affirmative. This conclusion is established in Theorem \ref{m} provided in the appendix. 

\section{Transversal Hausdorff Dimensions}\label{set:Transversal Hausdorff Dimensions}

In this section, we consider an $\mathbb{R}^k$ action $\alpha$ on $M$ by $C^{1+r}$ diffeomorphisms. 
Our goal is to prove Theorem \ref{THD}. 
It is worth mentioning that a similar but stronger result was also obtained  in \cite{BRW23}. 

Firstly, we have the Lyapunov decomposition of the tangent bundle, 
$$
TM=T\mathcal{O}\oplus\bigoplus_{i=1}^{L} E_{i}, 
$$
where $E_{i}$ is the Lyapunov subspace with respect to $\chi_i$. 
For each $t\in\mathbb{R}^k$, there is an order for the positive exponents 
$\chi_i({\bf t})$, $1\leq i\leq u({\bf t})$, and the corresponding subspaces $E_i$, 
such that for every $1\leq j\leq u({\bf t})$, the distribution 
$\displaystyle \bigoplus_{1\leq i\leq j}E_i$ is integrable. 
Similarly for the negative exponents. 
Below in this section, we will ignore $k$ zero exponents from the direction of flow.

Next, we prove the following slightly generalized proposition of \cite[Proposition 8.1.]{Hu93}.

\begin{proposition}\label{p1}
Let $f$ and $g$ be commuting $C^{1+r}$ diffeomorphisms on $M$
that preserve an ergodic measure $\mu$.  
Let 
$$
\chi_1(f)>\chi_2(f)>\cdots>\chi_{u(f)}(f)>0>\cdots>\chi_{m(f)}(f)
$$
denote all distinct Lyapunov exponents of $f$, possibly there are extra zero exponents. 
For $g$, we similarly have 
$$
\chi_1(g)>\chi_2(g)>\cdots>\chi_{u(g)}(g)>0>\cdots>\chi_{m(g)}(g).
$$ 
Suppose there exists $i$ with $1\leq i\leq \min\{u(f),u(g)\}$ such that 
$\displaystyle \bigoplus_{1\leq j\leq i}E_{\chi_j(f)}=\bigoplus_{1\leq j\leq i}E_{\chi_j(g)}$, 
Assume further there exist $ \lambda(f), \lambda(g)>0$ satisfying 
$$
\chi_{i+1}(f)< \lambda(f)< \chi_i(f),\quad  \chi_{i+1}(g)< \lambda(g)< \chi_i(g). 
$$
Then there exists a measurable partition $\xi$ on $M$ with the following properties: 
\begin{itemize}
\item[(1)] $\xi$ is subordinate to $W^i$, where $W^i$ is integrated by $\displaystyle \bigoplus_{1\leq j\leq i}E_{\chi_j(f)}=\bigoplus_{1\leq j\leq i}E_{\chi_j(g)}$;
\item[(2)] $\xi$ is an increasing partition for both $f$ and $g$;
\item[(3)] Both $\displaystyle\bigvee^{\infty}_{n=0}f^{-n}\xi$ and $\displaystyle\bigvee^{\infty}_{n=0}g^{-n}\xi$ are the partition into points (mod $\mu$);
\item[(4)] The biggest $\sigma$-algebra contained in $\displaystyle \bigcap^{\infty}_{n=0}\bigcap^{\infty}_{m=0}f^ng^m\xi$ is $\mathcal{B}^i$.
\end{itemize}
\end{proposition}

Here, we say $\xi$ is subordinate to $W^i$-foliation if for $\mu$-a.e. 
$x\in M$, $\xi(x)\subset W^i(x)$ and $\xi(x)$ contains a neighborhood of $x$ that is open 
in the submanifold topology of $W^i(x)$. 
A partition $\xi_1$ is said to refine $\xi_2$ 
(denoted by $\xi_1>\xi_2$) if for $\mu$-a.e. $x\in M$, 
we have $\xi_1(x)\subset\xi_2(x)$. 
A partition $\xi$ is called increasing if $\xi>f\xi$. 
Let $\mathcal{B}^i$ denote the 
sub-$\sigma$-algebra of the Borel $\sigma$-algebra on $M$ whose elements 
are unions of entire $W^i$-leaves. 
For further information, we refer the reader to \cite{LY-85-annals-2}.

\begin{proof}
The proof is similar to that of Proposition 8.1 in \cite{Hu93} with several modifications. 
First, for Lemma 8.3 in \cite{Hu93}, we consider the inequality 
$$
d^i(f^{-n}g^{-k}y, \partial B(x, \rho))e^{n(\lambda(f)-2\varepsilon)+k(\lambda(g)-2\varepsilon)}<b^{-1}.
$$ 
Second, for Lemma 8.4 in \cite{Hu93}, we prove the inequality 
$$
d^i(f^{-n}g^{-k}y, f^{-n}g^{-k}z)\leq 2Kl(z)d^i(y, z)e^{-n(\lambda(f)-2\varepsilon)-k(\lambda(g)-2\varepsilon)}.
$$
For the proof there, replacing $W^u_{\alpha}$ with $W^i_\alpha$ and $d_\omega$ with $d^i$ 
makes the same argument applicable to our case. 
We omit the detailed proof here for simplicity.
\end{proof}

In fact, the above proposition can be applied to the splitting that appears in the Lyapunov decomposition for $\alpha$, 
since the splitting in the proposition (or for two diffeomorphisms) 
is coarser than this one. And this is what we really need!

Given the partition $\xi$ defined above, we have the following proposition. 

\begin{proposition}\label{p3}
$H_\mu(\xi|fg\xi)=H_\mu(\xi|f\xi)+H_\mu(\xi|g\xi).$
\end{proposition}

\begin{proof}
We compute the conditional entropy as follows:
  \begin{align*}
  H_\mu(\xi|fg\xi)&=H_\mu(\xi\vee g\xi|fg\xi)=H_\mu(g\xi|fg\xi)+H_\mu(\xi|g\xi\vee fg\xi)\\
  &=H_\mu(\xi|f\xi)+H_\mu(\xi|g\xi).
\end{align*}
\end{proof}

Note also that we have $H_\mu(\xi|f\xi)=h_i(f)$ for any such partition 
(a result of section 9 of \cite{LY-85-annals-2}). Combining $h_i(f)=\sum_{j=1}^i \gamma_j(f) \chi_j(f),$ we have $H_\mu(\xi|f\xi)=\sum_{j=1}^i \gamma_j(f) \chi_j(f)$. This also applies to $g$ and $fg$, then Proposition \ref{p3} implies, 
\begin{equation}\label{e1}
\sum_{j=1}^i \gamma_j(fg) \chi_j(fg)=\sum_{j=1}^i \gamma_j(f) \chi_j(f)+\sum_{j=1}^i \gamma_j(g) \chi_j(g).
\end{equation}

Now, we are ready to prove Theorem \ref{THD}. 

\begin{proof}[Proof of Theorem \ref{THD}]
We split the proof of Theorem \ref{THD} into the following four parts. Below $n\in\mathbb{N}^+$, $r\in\mathbb{R}^+$, and ${\bf t}\in\mathbb{R}^k$ with ${\bf t}\neq 0$.

$(1)$ $\gamma_i(r{\bf t})=\gamma_i({\bf t})$. 

Considering the measurable partition $\xi$ (depending on $i$) 
built in Lemma 9.1.1 in \cite{LY-85-annals-2} for $\alpha({\bf t})$, 
then it is also a partition for $\alpha(n{\bf t})$ satisfying the same conditions. 
Combining $H_\mu(\xi|\alpha(n{\bf t})\xi)=nH_\mu(\xi|\alpha({\bf t})\xi)$ and $H_\mu(\xi|\alpha({\bf t})\xi)=h_i(\alpha({\bf t}))$, 
for $i=1, \cdots, u(t)$, then we get that $\gamma_i(n{\bf t})=\gamma_i({\bf t})$ for $\chi_i({\bf t})>0$. 
This also gives us that for any positive rational numbers $u$, 
$\gamma_i(u{\bf t})=\gamma_i({\bf t})$ for $\chi_i({\bf t})>0$. Now pick arbitrary ${\bf s},{\bf t}\in\mathbb{R}^k$ with ${\bf s}=r{\bf t}$ for some $r$, 
then 
$$ 
\gamma_1({\bf s}+{\bf t}) \chi_1({\bf s}+{\bf t})= \gamma_1({\bf s}) \chi_1({\bf s})+\gamma_1({\bf t}) \chi_1({\bf t}). 
$$ 
Hence we have  
$$
\gamma_1({\bf t})-\gamma_1({\bf s}+{\bf t})=(\gamma_1({\bf s}+{\bf t})-\gamma_1({\bf s}))r.
$$ 
If at least one of $(\gamma_1({\bf t})-\gamma_1({\bf s}+{\bf t}))$ and $(\gamma_1({\bf s}+{\bf t})-\gamma_1({\bf s}))$ is not $0$, 
then we can replace $r$ by $nr$ (arbitrary $n>0$), then we will get a contradiction, 
because all $\gamma_1$ are bounded by $\dim{E_1}$. Hence 
$$
\gamma_1({\bf t})=\gamma_1({\bf s}+{\bf t})=\gamma_1({\bf s}),
$$ 
and this completes the first step. The same argument works for the subsequent $\gamma_i$s.

$(2)$ We consider in one Weyl Chamber $\mathcal{C}$. 

Assume there are $u$ positive exponents. Dividing $\mathcal{C}$ by hyperplanes 
$$
L_{i,j}:=\{{\bf t}: \chi_i({\bf t})-\chi_j({\bf t})=0\}
$$ 
into some small sub-chambers. 
In each sub-chamber, the positive exponents have a fixed order 
(this order does not change as ${\bf t}$ varies). Thus, we can apply Proposition \ref{p3} 
and use induction on $i$. From Equation (\ref{e1}), when $i=1$, 
for every ${\bf s}, {\bf t}$ in that sub-chamber, we have 
$$
\gamma_1({\bf t}+{\bf s})\chi_1({\bf t}+{\bf s})=\gamma_1({\bf t})\chi_1({\bf t})+\gamma_1({\bf s})\chi_1({\bf s}),
$$ 
hence 
$$ 
(\gamma_1({\bf t})-\gamma_1({\bf s}+{\bf t}))\chi_1({\bf t})=(\gamma_1({\bf s}+{\bf t})-\gamma_1({\bf s}))\chi_1({\bf s}).
$$ 
If at least one of $(\gamma_1({\bf t})-\gamma_1({\bf s}+{\bf t}))$ and $(\gamma_1({\bf s}+{\bf t})-\gamma_1({\bf s}))$ is not $0$, 
then we can let ${\bf s}$ or ${\bf t}$ go to $\infty$, 
then we will get a contradiction due to the same reason in part (1). 
Hence $$\gamma_1({\bf t})=\gamma_1({\bf s}+{\bf t})=\gamma_1({\bf s}),$$ and this finished the first step. 
Suppose for $i<u$, we have $\gamma_j({\bf t}), j\leq i$ are all constant for all $t$ in the sub-chamber. 
Then consider Equation (\ref{e1}) for $i+1$, since the first $i$ THDs are equal, 
this will leave us 
$$
\gamma_{i+1}({\bf t}+{\bf s})\chi_{i+1}({\bf t}+{\bf s})=\gamma_{i+1}({\bf t})\chi_{i+1}({\bf t})+\gamma_{i+1}({\bf s})\chi_{i+1}({\bf s}).
$$ 
Use the argument in the first step, we get the desired result. 
Hence for all positive exponents, 
we have THDs are constant. The same holds for the negative exponents if 
we just consider the negative of the sub-chamber. 

The argument also applies to 
 points on a hyperplane that does not 
intersect any other hyperplanes or Lyapunov hyperplane.

$(3)$ We consider one Weyl Chamber $\mathcal{C}$. 

We consider two adjacent sub-chambers, 
$\mathcal{C}_1$ and $\mathcal{C}_2$ (where adjacent means separated only by one hyperplane). 
Note that multiple hyperplanes may coincide;  
if not, we can skip the following and go to next paragraph. 
Let's assume $L_{i,j}$ and $L_{p,q}$ are two of them, where $i,j,p,q$ are four different numbers, 
one can easily get that, on the hyperplane 
$$
\chi_i=\chi_j>0,\;\chi_p=\chi_q>0,\; \chi_i\neq\chi_p.
$$ 
Thus, the paired exponents will take different values on the hyperplane, 
and hence we can always consider them one by another, 
ordered from the paired exponents that take the greatest value to the paired exponents 
take the least value. Of course, we also need to take other positive exponents into account, 
which can be tackled by argument from part $(2)$. 

Without loss of generality, we suppose the first hyperplane is 
$$L_{i,j}=\{{\bf t}: \chi_i({\bf t})-\chi_j({\bf t})=0\}.$$ 
Consider ${\bf s}\in\mathcal{C}_1$ and ${\bf t}\in\mathcal{C}_2$, 
they are very close to $L_{i,j}$ and comparably far away from other hyperplanes or 
Lyapunov hyperplanes. For such ${\bf s}, {\bf t}$, $\chi_i$ and $\chi_j$ are two closed exponents 
in both sub-chambers, and only these two exponents will change order. 
On hyperplane $L_{i,j}$, they coincide. Suppose $\chi_i, \chi_j$ locate at $k, k+1$ in the order, 
then for the first $k-1$ exponents, the THDs are constant by argument in part (2). 
By apply Proposition \ref{p3} for $k+1$, we cancel the first $k-1$ exponents, then get 
\begin{align}
\gamma_i({\bf s}+{\bf t})\chi_i({\bf s}&+{\bf t})+\gamma_j({\bf s}+{\bf t})\chi_j({\bf s}+{\bf t})\nonumber\\
&=\gamma_i({\bf s})\chi_i({\bf s})+\gamma_j({\bf s})\chi_j({\bf s})+\gamma_i({\bf t})\chi_i({\bf t})+\gamma_j({\bf t})\chi_j({\bf t}).\label{e2}
\end{align}
When one of ${\bf s}, {\bf t}$ lies in $L_{i,j}$, Equation (\ref{e2}) still holds. Suppose ${\bf t}\in L_{i,j}$, this will give us $$\gamma_i({\bf s})+\gamma_j({\bf s})=\gamma_i({\bf t})+\gamma_j({\bf t})$$ for ${\bf s}$ in either $\mathcal{C}_1$ or $\mathcal{C}_2$. Hence from this, when ${\bf s},{\bf t}$ in different sub-chambers, $$\gamma_i({\bf s})+\gamma_j({\bf s})=\gamma_i({\bf t})+\gamma_j({\bf t}).$$

Now, suppose ${\bf s}\in \mathcal{C}_1,\,{\bf t}\in \mathcal{C}_2$ and ${\bf s}+{\bf t}\in \mathcal{C}_1$, Equation (\ref{e2}) is $$(\gamma_i({\bf s})-\gamma_i({\bf t}))\chi_i({\bf t})=(\gamma_j({\bf t})-\gamma_j({\bf s}))\chi_j({\bf t}).$$
Since $\chi_i({\bf t})\neq \chi_j({\bf t})$, we conclude that 
$$
\gamma_i({\bf s})=\gamma_i({\bf t}),\quad \gamma_{j}({\bf s})=\gamma_{j}({\bf t}).
$$

For other positive exponents, arguments in part (2) and the above work similarly. Hence the constantness of THDs can be proved when crossing the hyperplanes. All the above arguments can be also applied to one hyperplane when crossing some other hyperplane. And these show that the THDs are all constant in one Weyl Chamber.

$(4)$ We consider the case when crossing the Lyapunov hyperplane. 

There may be several exponents changing their sign. 
However, we do not need to consider these exponents, 
instead we only consider those exponents remain to be positive. 
The argument in (3) works in this case. We omit the details here.

Hence we complete the proof of Theorem \ref{THD}. 
\end{proof}

For future use, 
we denote $\gamma_i({\bf t})$ by $\gamma_i$.

\begin{remark}
One may easily figure out that, for maximal rank actions, say Cartan actions on tori, 
$E_{\chi_i}$ is integrable to some $\mathcal{W}^i$ for each $i$, 
and the corresponding THD $\gamma_i$ is, 
in fact the pointwise dimension of the conditional measure of $\mu$ restricted to $\mathcal{W}^i$. 
In this case, we would rather call $\gamma_i$ conditional dimension instead of 
transversal dimension! However, in the more general cases, especially 
when there are positive proportional exponents, some $\gamma_i$ really represents 
the dimension of the transversal direction rather than conditional dimension.
\end{remark}

\section{Main Reduction}\label{set:Main Reduction}

The following three sections are dedicated to the proof of Theorem \ref{bk}. 
In this section, we restate the theorem as Theorem \ref{bk} and give a reduction 
from abelian action to one diffeomorphism case. 
The complete proof of Theorem \ref{bk} is divided between these two subsequent sections: 
the next section deals with the hyperbolic case, 
while the section after that addresses the absolutely continuous case.

As established in \S \ref{subset:Lyapunov Exponents, Suspension, Charts}, 
the suspension of a $\mathbb{Z}^k$ action can naturally identified as an $\mathbb{R}^k$ action. 
In this section, we focus on an action $\alpha: \mathbb{R}^d \times (M, \mu) \to (M, \mu)$ 
that preserves an ergodic measure $\mu$, 
and denote its corresponding Lyapunov exponent functionals by 
$\mathcal{L} = \{\lambda_i: 1 \leq i \leq L\}$. 

%Since $\mu$ is ergodic, it is easy to see that, $$ \lim_{\epsilon\to 0}\limsup_{s\to \infty}\frac{-\log(\mu(B(\alpha, F^p_s, x, \epsilon)))}{s}\;\text{and} \; \lim_{\epsilon\to 0}\liminf_{s\to \infty}\frac{-\log(\mu(B(\alpha, F^p_s, x, \epsilon)))}{s}$$ are constant a.e. We will denote the two constant by $\mathcal{D}_\mu$ and $\mathcal{E}_\mu$ respectively. Basically, we are going to prove the following two inequalities:
%\begin{equation}\label{eq1}
%\mathcal{D}_\mu\leq \sum_{i=1}^{D}\gamma_i \max_{{\bf t}: p({\bf t})\leq 1} \chi_i({\bf t}),
%\end{equation}
%and
%\begin{equation}\label{eq2}
%\mathcal{E}_\mu\ge \sum_{i=1}^{D}\gamma_i \max_{{\bf t}: p({\bf t})\leq 1} \chi_i({\bf t}).
%\end{equation}

%The property (4) implies that no two distinct exponents in the Lyapunov decomposition of 
%$f$ are equal.
% It is worth mentioning that 
%there may exist a Lyapunov exponent $\lambda_i$ of $f$ 
%corresponding to a Lyapunov subspace $E^i$ with $\operatorname{dim}E_i\geq 2$. 
%In this case, for any $t\in\mathbb{R}^k$, 
%the Lyapunov exponents of $\alpha({\bf t})$ restricted to $E_i$ coincide. 

Ben Ovadia and Rodriguez Hertz established the Besicovitch-Bowen covering lemma 
and the Bowen-Lebesgue density lemma in \cite{ORH23}. 
A similar version of the Bowen-Lebesgue density lemma can be found in \cite[Lemma 2.2]{DQ25}. 
Adapting the strategy of \cite{ORH23, DQ25}, we obtain the following lemmas. 
Recall that $m$ denotes the dimension of the manifold $M$. 

\begin{lemma}\label{ORH23-Lemma 2.2}
  Fix a small $\epsilon$ and a large $l$, 
  and let $\Phi_{x_0}$ denote the Lyapunov chart at $x_0\in \Gamma_{l,\epsilon}$. 
 Let $ A \subset \Phi_{x_0} \cap \Gamma_{l,\epsilon} $ be a measurable subset. 
 Then $ A $ can be covered by Bowen balls $B(\alpha, F^p_n,\cdot,\epsilon)$ centered at points of $ A $ 
 with the covering multiplicity bounded by $ e^{3 nm\epsilon} $, 
 where $ n $ is sufficiently large with respect to $ \Gamma_{l,\epsilon} $. 
 \end{lemma}

 \begin{lemma}\label{ORH23-Lemma 2.3}
Fix $\epsilon>0$, let $\mu$ be an ergodic invariant Borel probability measure, and let $ A $ be a measurable set with $\mu(A)>0 $. 
Then for $ \mu $-a.e. $ x \in A $,
  \[
    \lim_{n \rightarrow \infty} -\frac{1}{n} \log \frac{\mu\left(B(\alpha, F^p_n, x, \epsilon)\cap A\right)}{\mu\left(B(\alpha, F^p_n, x, \epsilon)\right)}=0.
  \]
\end{lemma}

Additionally, if the set $A$ in Lemma \ref{ORH23-Lemma 2.2} is assumed to be contained in an element of a measurable partition, 
then by modifying the proofs, we conclude that Lemmas \ref{ORH23-Lemma 2.2} and 
Lemma \ref{ORH23-Lemma 2.3} also holds for conditional measures. 

The proofs of these lemmas are almost identical to those in \cite{ORH23}, 
with the only modification that the $r$-neutralized Bowen ball $B(x,n,e^{-nr})$ is replaced by 
the Bowen ball $B(\alpha, F^p_n,x,\epsilon)$. Hence we omit the proofs. 

\iffalse
\begin{proposition}\label{main}
Under the same assumptions as in Theorem \ref{mainresult}, for $\mu$-a.e. $x$, 
\begin{align}
\lim_{\epsilon\to 0}\liminf_{n\to \infty}\frac{-\log(\mu(B(\alpha, F^p_n, x, \epsilon)))}{n}=\lim_{\epsilon\to 0}\limsup_{n\to \infty}\frac{-\log(\mu(B(\alpha, F^p_n, x, \epsilon)))}{n},
\end{align}
and this limit is equal to 
$$
\sum_{i=1}^{q}\gamma_i \max_{{\bf t}: p({\bf t})\leq 1} \chi_i({\bf t}).
$$
\end{proposition}
\fi

Next, we establish the proof of Theorem \ref{mainresult} by using Theorem \ref{bk}.

\begin{proof}[Proof of  Theorem \ref{mainresult}]
We define $\displaystyle \Delta:=\sum_{i=1}^{L}\gamma_i \max_{{\bf t}: p({\bf t})\leq 1} \chi_i({\bf t})$. 
By Theorem \ref{bk}, there exists a full measure set $M'$ 
such that the following holds: for any $\theta>0$, 
every $ x\in M'$ admits $\epsilon(x)>0, n(x)>0$ with the property that for all 
 $\epsilon<\epsilon(x)$ and $n>n(x)$, 
$$
e^{-n(\Delta+\theta)}\leq \mu(B(\alpha, F^p_n, x,\epsilon))\leq e^{-n(\Delta-\theta)}.
$$ 
For each positive integer $n$, we define 
$$ 
M_{n}':=\left\{x\in M': \epsilon(x)\ge \frac{1}{n},\, n(x)\leq n\right\}.
$$ 
Then $\displaystyle M'=\bigcup_n M_n'$. For any $\delta>0$, there exists $n_0>0$ 
such that $\mu(M_{n_0}')>1-\frac{\delta}{2}.$ Furthermore, within $M_{n_0}'$, 
there exists a compact subset $L$ satisfying $\mu(L)>1-\delta.$
Without loss of generality, we can assume the set $L$ is restricted to Lyapunov charts. 

Consider the minimal number $\#B_n$ of Bowen balls $B(\alpha, F^p_n,\cdot,\epsilon)$ covering $L$. 
On one hand, it is easy to see that 
\begin{align}\label{slow entropy-(1)}
\#B_n\ge \frac{1-\delta}{\max_{x\in L}{\mu(B(\alpha,F^p_n,x,\epsilon))}}\ge (1-\delta)e^{n(\Delta-\theta)}.
\end{align} 
On the other hand, take a set of points $\{x_n\}$ in $L$ such that 
\begin{align*}
\max_{p({\bf t})\leq n}d(\alpha({\bf t})x_i,\alpha({\bf t})x_j)\ge \epsilon,\quad \text{ for any }i\neq j. 
\end{align*}
We can choose such a set with the maximal number of elements; 
denote this number by $\#M_n$ and the set by $\Xi$. 
Then $\{B(\alpha,F^p_n,x,\epsilon)\}_{x\in \Xi}$ covers $L$. 
For sufficiently large $n$, by Lemma \ref{ORH23-Lemma 2.2}, 
we conclude that every $ x\in M$ is covered by at most $C_de^{n\theta}$ Bowen balls, 
where $C_d$ is a constant only depend on dimension $d$. 
Thus, we have 
\begin{align}\label{slow entropy-(2)}
\#B_n\leq \#M_n\leq \frac{C_de^{n\theta}}{\min_{x\in \Xi}{\mu(B(\alpha,F^p_n,x,\epsilon))}}\leq C_de^{n(\Delta+2\theta)}.
\end{align}
Combining inequalities \eqref{slow entropy-(1)} and \eqref{slow entropy-(2)}, 
and noting that we can let $\theta\to 0$ as $\epsilon\to 0$, 
we have 
$$
sh_{\mu}(\alpha,p)=\lim_{\epsilon, \delta \to 0}\limsup_{n\to\infty}\frac{\log{\#B_n}}{n}=\Delta=\sum_{i=1}^{L}\gamma_i \max_{{\bf t}: p({\bf t})\leq 1} \chi_i({\bf t}).
$$ 
This completes the proof of Theorem \ref{mainresult}.
\end{proof}

\begin{remark}
From the above argument, we can easily see that 
$$
\lim_{\epsilon, \delta \to 0}\limsup_{n\to\infty}\frac{\log{\#B_n}}{n}=\lim_{\epsilon, \delta \to 0}\liminf_{n\to\infty}\frac{\log{\#B_n}}{n}.
$$ 
Hence in the definition of slow entropy, we obtain
$$
sh_{\mu}(\alpha, p)=\lim_{\epsilon, \delta \to 0}\limsup_{n\to\infty}\frac{\log{S_d (\alpha, F^p_n, \epsilon, \delta)}}{n}=\lim_{\epsilon, \delta \to 0}\liminf_{n\to\infty}\frac{\log{S_d (\alpha, F^p_n, \epsilon, \delta)}}{n}.
$$
\end{remark}

Before proving Theorem \ref{bk}, we first introduce the following definitions and results as preliminaries. 

The following lemma describes the selection of a diffeomorphism from the abelian action.

\begin{lemma}\label{choose f}
There exists a  ${\bf t}\in \mathbb{R}^k$ such that $f:=\alpha({\bf t})$ satisfies the following properties: 
\begin{itemize}
\item[(1)] $p({\bf t})\leq 1$; 
\item[(2)] $\mu$ is ergodic with respect to $f$; 
\item[(3)] there is no extra zero exponent, i.e., 
for any ${\bf s}\in \mathbb{R}^k$, 
 the Lyapunov exponent $\chi_i({\bf s})$ of $\alpha({\bf s})$ 
 vanishes if the corresponding Lyapunov exponent of $\chi_i(f)$ vanishes; 
\item[(4)] for any ${\bf s}\in \mathbb{R}^k$, 
 the Lyapunov exponents in the decomposition of $\alpha({\bf s})$ relative to $f$ are distinct.
\end{itemize}
\end{lemma}

\begin{proof}
We only need to focus on the Lyapunov exponents $\chi_i$ for which there exists ${\bf s}\in\mathbb{R}^k$ such that $\chi_i({\bf s})$ does not vanish. Therefore it suffices to prove that property (4) holds for some ${\bf t}\in \mathbb{R}^k$. 
Fix ${\bf t}_0\in \mathbb{R}^k$. Suppose there exists ${\bf t}_1\in \mathbb{R}^k$ such that 
two distinct Lyapunov exponents in the decomposition of $\alpha({\bf t}_1)$ 
relative to $\alpha({\bf t}_0)$ coincide. 

First, we consider the following simple case; 
the general case can be handled by the same argument. 
For some positive integer $q$, assume that 
\[
\chi_0^g>\chi_1^g>\chi_2^g>\cdots>\chi_{q}^g 
\quad \text{and}\quad 
\chi_1^f>\chi_2^f>\cdots>\chi_{q}^f,
\]
where the direct sum of the Lyapunov subspace for $\chi_0^g,\chi_1^g$ 
are $E_1$, the Lyapunov subspace of $\chi_1^f$,  
and $\chi_i^g$ and $\chi_i^f$ belong to the same Lyapunov subspace $E_i$ for $2\leq i\leq q$. 

Set 
\[
\epsilon_0:=\frac{1}{100}\min\Big\{\chi_q^g,\,\chi_q^f,\,\chi_{i}^f-\chi_{i+1}^f,\,\chi_{j}^g-\chi_{j+1}^g: \,1\leq i<q, \,0\leq j<q\Big\}.
\]
Then there exists $0<\epsilon<\epsilon_0$ such that 
the Lyapunov exponents of $\alpha({\bf t}_0+\epsilon {\bf t}_1)$ are
\[
\chi^f_1+\epsilon \chi_0^g 
> \chi^f_1+\epsilon \chi_1^g 
> \chi^f_2+\epsilon \chi_2^g 
> \cdots 
> \chi^f_q+\epsilon \chi_q^g.
\]
In the general case, the argument is similar: for sufficiently small $\epsilon$, 
only finitely many values of $\epsilon$ can cause coincidences among the Lyapunov exponents of $\alpha({\bf t}_0+\epsilon {\bf t}_1)$. 
Thus, by replacing $f$ with $\alpha({\bf t}_0+\epsilon {\bf t}_1)$ and repeating this process if necessary, 
we obtain $f$ such that for any ${\bf t}\in \mathbb{R}^k$, 
the Lyapunov exponents in the decomposition of $\alpha({\bf t})$ relative to $f$ are distinct.

This completes the proof of the lemma. 
\end{proof}

\begin{remark}
Obviously, if the ergodic invariant measure $\mu$ is hyperbolic, then property (3) implies that no nontrivial exponent of $\alpha$ vanishes for $f$. 
\end{remark}

Given the diffeomorphism $f:=\alpha({\bf t})$ in Lemma \ref{choose f}, 
let $u$ be the dimension of the unstable Lyapunov subspace and $s$ be that of the stable one. 
We denote the Lyapunov exponents of the ergodic invariant measure $\mu$ (corresponding to the nontrivial exponents of $f$) by  
\begin{align*}
\chi_1>\cdots>\chi_{u}>\chi_{u+1}=0\,\text{(possible!)}>\chi_{u+2}>\cdots> \chi_{L}, 
\end{align*}
where $L=u+s+1$ and $E_i$ are the corresponding Lyapunov subspaces with  $d_i:=\operatorname{dim} E_i $. 
We fix the order of the exponents as this once and for all.  Let $W^i$ be the $i$-foliation integrated by $\displaystyle \bigoplus_{1\le j\le i} E_j$ when $i\le u$, and by $\displaystyle \bigoplus_{L+1-i\le j\le L}E_j$ when $i\ge u+2$. 
Let $\xi_i$ be a measurable partition subordinate to $W^i$, and $\{\mu_x^i\}$ be a system of the induced conditional measures. For convenience, we use $\mu_x^s$ to denote $\mu_x^L$.

The following result assumes that $\alpha$ is a $C^{1+r}$ abelian $\mathbb{R}^k$-action with an ergodic invariant measure $\mu$, and the assumption that $\mu$ is hyperbolic or absolutely continuous is not used. 

\begin{lemma}\label{le1}
Let $\displaystyle a_i=\max_{{\bf t}: p({\bf t})\leq 1} \chi_i({\bf t})$ for $1\leq i\leq L$, take $\epsilon\leq \min\left\{\frac{1}{100ml},\epsilon_0\right\}$, then for any $x\in \Gamma_{l,\epsilon}$, there exists $s(x)>0$, such that when $s\ge s(x)$,
\begin{align*}
K^{-1}\bigg(\prod_{i\leq u}B_i\Big(0, \frac{\epsilon e^{-(a_i+3\epsilon)n}}{m+1}\Big)&\times B_{u+1}\Big(0, \frac{\epsilon e^{-3\epsilon n}}{m+1}\Big)
\times \prod_{i\ge u+2}B_i\Big(0, \frac{\epsilon e^{-(a_i+3\epsilon)n}}{m+1}\Big)\bigg)\\
&\subset \Phi_x ^{-1}\Big(B(\alpha, F^p_s, x, \epsilon)\Big)\subset 
\\l\bigg(\prod_{i\leq u}B_i\Big(0, (m+1)\epsilon e^{-(a_i-3\epsilon)n}\Big)&\times B_{u+1}\Big(0, (m+1)\epsilon\Big)\\
&\times \prod_{i\ge u+2}B_i\Big(0, (m+1)\epsilon e^{-(a_i-3\epsilon)n}\Big)\bigg).
\end{align*}
Here, $B_i$ is the ball centered at origin in $\mathbb{R}^{d_i}$, and $\prod$ denotes the usual direct product.
\end{lemma}

\begin{proof}
Note that, for the neutral direction, 
it will neither contract more than sub-exponentially nor expand more than $(m+1)\epsilon$. 
So we only need to prove the inclusion for the other directions. First, we prove the left hand side inclusion. It is enough to show that for any 
\begin{align*}
u\in \left\{K^{-1}\bigg(\prod_{i\leq u}B_i\Big(0, \frac{\epsilon e^{-(a_i+3\epsilon)n}}{m+1}\Big)\right.\times B_{u+1}\Big(0, \frac{\epsilon e^{-3\epsilon n}}{m+1}\Big)
\left.\times \prod_{i\ge u+2}B_i\Big(0, \frac{\epsilon e^{-(a_i+3\epsilon)n}}{m+1}\Big)\bigg) \right\}
\end{align*}
and any $t$ with $p({\bf t})\leq n$, we have $d(\alpha({\bf t})x,\alpha({\bf t})\Phi_x(u))\leq \epsilon.$

Combining the properties (iii) and (iv) in Proposition \ref{chart2} and the following type estimate 
\begin{align}\label{le1-(1)}
\left| f_1\circ g_1-f_2\circ g_2\right|\leq \left| f_1\circ g_1-f_1\circ g_2\right| + \left| f_1\circ g_2-f_2\circ g_2\right|,
\end{align}
we obtain 
\begin{align*}
 \left|\Phi_{\alpha({\bf t})x}^{-1}\alpha({\bf t})\Phi_x(u)-D\Phi_{\alpha({\bf t})x}^{-1}\alpha({\bf t})\Phi_x(0)u\right|
\leq K^{-1} \epsilon (n+1) \left(\sum_{i=1}^D\epsilon e^{n\epsilon}e^{a_i n+n\epsilon}\frac{ e^{-(a_i+3\epsilon)n}}{m+1}\right). 
\end{align*}
Combining this with property (iii) in Proposition \ref{chart2}, we have 
\begin{align*}
\left|\Phi_{\alpha({\bf t})x}^{-1}\alpha({\bf t})\Phi_x(u)\right|&
\leq \left|\Phi_{\alpha({\bf t})x}^{-1}\alpha({\bf t})\Phi_x(u)-D\Phi_{\alpha({\bf t})x}^{-1}\alpha({\bf t})\Phi_x(0)u\right|+\left|D\Phi_{\alpha({\bf t})x}^{-1}\alpha({\bf t})\Phi_x(0)u\right|\\
&\leq K^{-1} \epsilon (n+1)\left(\sum_{i=1}^D\frac{\epsilon e^{-\epsilon n}}{m+1}\right)+K^{-1}\left(\sum_{i=1}^{D}e^{a_i n+n\epsilon}\frac{\epsilon e^{-(a_i+3\epsilon)n}}{m+1}\right) \\
&\leq K^{-1}\left(\sum_{i=1}^D\frac{\epsilon^2 (n+1)e^{-\epsilon n}}{m+1}+\sum_{i=1}^{D}\frac{\epsilon e^{-2\epsilon n}}{m+1}\right).
\end{align*}
Hence when $n$ is sufficiently large, we conclude that
$$
|\Phi_{\alpha({\bf t})x}^{-1}\alpha({\bf t})\Phi_x(u)|\leq K^{-1}\epsilon$$ 
and it follows that by property (vi) in Proposition \ref{chart2}, we deduce 
$$
d(\alpha({\bf t})x,\alpha({\bf t})\Phi_x(u))\leq K|\Phi_{\alpha({\bf t})x}^{-1}\alpha({\bf t})\Phi_x(u)|\leq \epsilon.
$$

Now we come to the proof of the other side. Assume 
$$
u:=(u_1,u_2,\cdots,u_D)\in \Phi_x ^{-1}B(\alpha, F^p_n, x, \epsilon)
$$
where $u_i\in \mathbb{R}^{d_i}$. It is enough to show that $u_i\in lB_i\Big(0, (m+1)\epsilon e^{-(a_i-3\epsilon)n}\Big)$ for every $i\leq u$, and a similar result holds for the remaining $i$s. First, choose $t$ such that $\chi_i({\bf t})=na_i$, then by the same argument as for \eqref{le1-(1)} and the fact that $\alpha({\bf t})$ expands the distance along the unstable direction, we have 
\begin{align*}
\left|\Phi_{\alpha({\bf t})x}^{-1}\alpha({\bf t})\Phi_x(u)-D\Phi_{\alpha({\bf t})x}^{-1}\alpha({\bf t})\Phi_x(0)u\right|
\leq  \epsilon (n+1) e^{n\epsilon}\left|\Phi_{\alpha({\bf t})x}^{-1}\alpha({\bf t})\Phi_x(u)\right|
\end{align*}
Thus, we have 
\begin{align*}
\left|\Phi_{\alpha({\bf t})x}^{-1}\alpha({\bf t})\Phi_xu\right|\geq \frac{\left|D\Phi_{\alpha({\bf t})x}^{-1}\alpha({\bf t})\Phi_x(0)u\right|}{\epsilon (n+1) e^{n\epsilon}+1} \geq \frac{e^{(a_i-\epsilon)n}|u_i|}{\epsilon (n+1) e^{n\epsilon}+1}.
\end{align*}
By property (vi) in Proposition \ref{chart2}, we derive 
$$
\left|\Phi_{\alpha({\bf t})x}^{-1}\alpha({\bf t})\Phi_xu\right|\leq ld(\Phi_{\alpha({\bf t})x}(\Phi_{\alpha({\bf t})x}^{-1}\alpha({\bf t})\Phi_xu),\alpha({\bf t})x)\leq l\epsilon.
$$
Hence, when $n$ is sufficiently large, we obtain 
$$
\left|u_i\right|\leq \left|\Phi_{\alpha({\bf t})x}^{-1}\alpha({\bf t})\Phi_xu\right|\frac{\epsilon (n+1) e^{n\epsilon}+1}{e^{(a_i-\epsilon)n}}\leq l\epsilon \frac{\epsilon (n+1) e^{n\epsilon}+1}{e^{(a_i-\epsilon)n}}\leq l(m+1)\epsilon e^{-(a_i-3\epsilon)n}.
$$
Thus, we finish the proof of this lemma. 
\end{proof}

We continue with the notions of coarse Lyapunov exponents 
and coarse Lyapunov foliations introduced in \cite{BRW23}.

\begin{definition}[\cite{BRW23} Definition 4.9]
Two Lyapunov exponents $\lambda_i$ and $\lambda_j \in \mathcal{L}$ are equivalent 
if they are positively proportional; that is, 
if there exists a constant $c > 0$ such that $\lambda_i = c\lambda_j$. 
A coarse Lyapunov exponent is an equivalence class in $\mathcal{L}$.
\end{definition}

We use $\widehat{\mathcal{L}}$ to denote the set of coarse Lyapunov exponents. 

\begin{definition}[\cite{BRW23} Definition 4.10]
Given $\chi \in \widehat{\mathcal{L}}$ with $\chi \neq 0$, 
the coarse Lyapunov foliation corresponding to $\chi$ is defined by 
\begin{equation}
\mathcal{W}^\chi := \bigcap_{\{{\bf t} \in \mathbb{R}^d: \chi({\bf t}) > 0\}} \mathcal{W}_{{\bf t}}^u, 
\end{equation}
where $\mathcal{W}_{{\bf t}} ^u$ is the unstable foliation $W^u$ for $\alpha({\bf t})$. 
\end{definition}

Clearly, the coarse Lyapunov foliation $\mathcal{W}^\chi$ is a $C^{1+r}$ foliation. 
The coarse Lyapunov manifold corresponding to $\chi$ through $x$ 
is the leaf $\mathcal{W}^\chi(x)$. 
For $\delta>0$ and $x\in M$, we use $\mathcal{W}^\chi(x,\delta)$ to represent the ball centered at $x$ 
with respect to the induced distance on $\mathcal{W}^\chi(x)$. 
 Let $d_{\mathcal{W}^\chi}$ denote the metric induced by the
 Riemannian structure on leaves of $\mathcal{W}^\chi$.

For convenience, we use ``coarse unstable foliations'' and ``coarse stable foliations'' 
as shorthand for the coarse Lyapunov foliations corresponding to positive Lyapunov exponents 
and negative Lyapunov exponents, respectively.

\section{Hyperbolic Case}\label{set:Hyperbolic Case}

In this section, we give the proof of Theorem \ref{bk} in the case of hyperbolic measure.

When the ergodic invariant measure $\mu$ is hyperbolic, 
given the diffeomorphism $f:=\alpha({\bf t})$ in Lemma \ref{choose f}, 
we denote the coarse Lyapunov exponents of the measure $\mu$ (corresponding to the nontrivial exponents of $f$) by  
$$
\tilde{\chi}_1(f), \cdots, \tilde{\chi}_{u_0}(f), \tilde{\chi}_{u_0+1}(f), \cdots, \tilde{\chi}_{u_0+s_0}(f),
$$ 
where $\tilde{\chi}_{i}(f)$ for $1\leq i\leq u_0$ corresponds to the positive coarse Lyapunov exponents of $f$, 
and $\tilde{\chi}_{u+i}$ for $1\leq i\leq s_0$ corresponds to the negative coarse Lyapunov exponents for $f$. 
Here, $u_0$ and $s_0$ are integers. 
Each $\tilde{\chi}_j$ denotes the set $\{\tilde{\chi}_{j,1},\cdots \tilde{\chi}_{j,k(j)}\}$; 
by the definition of coarse Lyapunov exponents, 
the elements of $\{\tilde{\chi}_{j,1},\cdots \tilde{\chi}_{j,k(j)}\}$ are positively proportional to one another. 
For a general diffeomorphism $\alpha({\bf n})$ with ${\bf n}\in\mathbb{R}^k$ 
induced by the $\mathbb{R}^k$-action, 
we use $\tilde{\chi}_j(\alpha({\bf n}))$ and $\tilde{\chi}_{j,i}(\alpha({\bf n}))$ 
to denote the corresponding notions in a similar manner. 

For each $j\in \{1,\cdots,u_0,u_0+1,\cdots,u_0+s_0\}$, there exists a diffeomorphism $f_j=\alpha({\bf t}_j)$ satisfying $p({\bf t}_j)\leq 1$ 
such that the Lyapunov exponents of 
$f_j$ along the coarse Lyapunov foliation $W^{\tilde{\chi}_j}$
satisfy $\displaystyle \tilde{\chi}_{j,i}(f_j)=\max_{{\bf t}:p({\bf t})\leq 1}\tilde{\chi}_{j,i}({\bf t})$ for any $1\leq i\leq k(j)$. 
This property is ensured by the definition of coarse Lyapunov exponents. 

Following the construction used in \cite{LS82,LY-85-annals-1,LY-85-annals-2}, 
there exists a measurable subordinate partition $\xi_j$ 
which is subordinate to the coarse Lyapunov foliation $W^{\tilde{\chi}_j}$. 
For each $x\in M$, we define the upper pointwise dimension 
and the lower pointwise dimension with respect to $\mu$ by 
$$
d^{\tilde{\chi}_j}_\mu(x):=\limsup_{r\to 0}\frac{\log \mu^{\tilde{\chi}_j}_x\left(W^{\tilde{\chi}_j}(x,r)\right)}{\log r}, \quad \underline{d}^{\tilde{\chi}_j}_\mu(x):=\liminf_{r\to 0}\frac{\log \mu^{\tilde{\chi}_j}_x\left(W^{\tilde{\chi}_j}(x,r)\right)}{\log r}. 
$$
By applying the same definition in Subsection \ref{set:Transversal Hausdorff Dimensions}, 
we define the corresponding transversal Hausdorff dimension 
$\{\gamma_{j,i}\}_i$ (for $\tilde{\chi}_{j,i}$ in 
the coarse Lyapunov foliation $W^{\tilde{\chi}_j}$ for $1\leq i\leq k(j)$). 
For each $j$ and $x\in M$, 
the upper local slow entropy and the lower local slow entropy 
(for the coarse Lyapunov foliation $W^{\tilde{\chi}_j}$) are defined respectively by 
\begin{align*}
sh_j(x):=\lim_{\epsilon\to 0}\limsup_{n\to\infty}-\frac{\log{\mu^{\tilde{\chi}_j}_x(B(\alpha, F^p_n, x, \epsilon))}}{n}
\end{align*}
and 
\begin{align*}
\underline{sh}_j(x):=\lim_{\epsilon\to 0}\liminf_{n\to\infty}-\frac{\log{\mu^{\tilde{\chi}_j}_x(B(\alpha, F^p_n, x, \epsilon))}}{n}. 
\end{align*}

We present a modified version of the Ledrappier-Young entropy formula 
with respect to the coarse Lyapunov foliatiion for our purposes here.

\begin{lemma}\label{entropy formula}
Given $1\leq j\leq u_0$, for $\mu$-almost every $x\in M$, 
\begin{align*}
sh_j(x)=\underline{sh}_j(x)=\sum_{i=1}^{k(j)} \gamma_{j,i} \max_{{\bf t}:p({\bf t})\leq 1}\tilde{\chi}_{j,i}({\bf t})
\end{align*}
and the same conclusion holds for $u_0+1\leq j\leq u_0+s_0$. 
\end{lemma}

\begin{proof}
Although the diffeomorphisms studied in \cite{LY-85-annals-1, LY-85-annals-2} 
are considered under the $C^2$ hypothesis, 
Brown pointed out in \cite{Bro22} that the unstable holonomies within center-unstable sets 
are Lipschitz continuous. This verifies that the Ledrappier-Young entropy formulas hold 
for $C^{1+r}$ diffeomorphisms. 

  Given $x\in M$, $n\in\mathbb{N}$ and $\delta>0$, we define the Bowen ball for $f_j$ with respect to the foliation $ W^{\tilde{\chi}_j}$ by 
\begin{align*}
W^{\tilde{\chi}_j}(f_j, x,n,\delta):=\left\{ y\in W^{\tilde{\chi}_j}(x): d_{W^{\tilde{\chi}_j}}(f_j^ix, f_j^iy)<\delta, \,\forall\, 0\leq i<n\right\}. 
\end{align*}
Recall the definition of $f_j$, for a fixed $\epsilon>0$, 
  there exists a function $n_1:M\to \mathbb{N}$ such that for $\mu$-a.e. $x$ and $n\geq n_1(x)$, 
\begin{align}\label{compare Bowen}
W^{\tilde{\chi}_j}(f_j, x,n,e^{-n\epsilon})\subset W^{\tilde{\chi}_j}(x)\cap B(\alpha, F^p_n, x, \epsilon)\subset W^{\tilde{\chi}_j}(f_j, x,n,e^{n\epsilon})
\end{align}
Following the arguments in \cite{LY-85-annals-2}, 
for $\mu$-a.e. $x\in M$, we have 
\begin{align*}
\sum_{i=1}^{k(j)} \gamma_{j,i} \left(\max_{{\bf t}:p({\bf t})\leq 1}\tilde{\chi}_{j,i}({\bf t})+\epsilon\right)=\lim_{n\to\infty}-\frac{\log{\mu^{\tilde{\chi}_j}_x(W^{\tilde{\chi}_j}(f_j, x,n,e^{-n\epsilon}))}}{n},\\
\sum_{i=1}^{k(j)} \gamma_{j,i} \left(\max_{{\bf t}:p({\bf t})\leq 1}\tilde{\chi}_{j,i}({\bf t})-\epsilon\right)= \lim_{n\to\infty}-\frac{\log{\mu^{\tilde{\chi}_j}_x(W^{\tilde{\chi}_j}(f_j, x,n,e^{n\epsilon}))}}{n}. 
\end{align*}
Combining this with \eqref{compare Bowen} and the arbitrariness of $\epsilon$, 
we finish the proof of this lemma. 
\end{proof}

The above result can be derived from the results in \cite{O24, DQ25}.

\begin{remark}
  Clearly, $sh_j(x)$ and $\underline{sh}_j(x)$ are $f$-invariant functions. 
  By Lemma \ref{entropy formula} and the ergodicity, these functions are constant. 
We denote this constant value by $sh_j$. 
In fact, this result holds for any general ergodic invariant meausre with respect to 
coarse unstable foliations and coarse stable foliations. 
\end{remark}

\begin{lemma}\label{entropy formula-1.5}
Fix $l>1$. 
For any $\epsilon, \delta>0$ with $\epsilon<\delta\leq \min\left\{\frac{1}{100ml},\epsilon_0 \right\}$, 
for every $x\in \Gamma_{l,\delta}$, there exists a constant $b=b(x,l,\epsilon,\delta)>0$ 
such that for any integer $n\gg 1$, 
\begin{align*}
B(\alpha, F^p_n, x,\epsilon)\subset  B(\alpha, F^p_n, x,\delta) \subset B(\alpha, F^p_{n-b}, x,\epsilon).
\end{align*}
\end{lemma}

\begin{proof}
Under the assumption of hyperbolic measure, 
for convenience, we still use the index $i$ as presented in Lemma \ref{le1}, 
even though the case $i=u+1$ is trivial. 
For any $1\leq i\leq u$ and $u+2\leq i\leq u+s$, we have $a_i>0$. 

Clearly, it suffices to prove that $ B(\alpha, F^p_b, x,\delta) \subset B(x,\epsilon)$ for any $x\in \Gamma_{l,\delta}$. 
For each $x\in\Gamma_{l,\delta}$, there exists an integer $b=b(x,l,\epsilon, \delta)\geq s(x)$ such that 
$$
l(m+1)\delta e^{-(a_i-3\epsilon)b}\leq K^{-1}\epsilon, \quad \text{ for any } 1\leq i\leq u \text{ and } u+2\leq i\leq u+s. 
$$
Combining this with Proposition \ref{chart2} (vi), we have 
\begin{align*}
 &\Phi_x ^{-1}\Big(B(\alpha, F^p_b, x, \delta)\Big)\\
 \subset& l\bigg(\prod_{i\leq u}B_i\Big(0, (m+1)\delta e^{-(a_i-3\epsilon)b}\Big) \times \prod_{i\ge u+2}B_i\Big(0, (m+1)\delta e^{-(a_i-3\epsilon)b}\Big)\bigg)\\
 \subset& \prod_{i\leq u}B_i\Big(0, K^{-1}\epsilon \Big) \times \prod_{i\ge u+2}B_i\Big(0, K^{-1}\epsilon \Big)\\
 \subset& \Phi_x ^{-1}\Big( B(x, \epsilon)\Big).
\end{align*}
Therefore, we deduce $B(\alpha, F^p_b, x, \delta)\subset B(x,\epsilon)$ for any $x\in \Gamma_{l,\delta}$. 
This implies that $B(\alpha, F^p_n, x,\delta) \subset B(\alpha, F^p_{n-b}, x,\epsilon)$ holds for such $x$. 
Since the other inequality is obvious, we complete the proof of the lemma. 
\end{proof}

Fix $l>1$ and $\delta\leq \min\left\{\frac{1}{100ml},\epsilon_0 \right\}$. 
Combining Lemma \ref{entropy formula-1.5} with Lemma \ref{entropy formula}, 
we conclude that for any $\epsilon>0$ with 
$\epsilon<\delta$, $1\leq j\leq u_0+s_0$, and any $x\in \Gamma_{l,\delta}$, 
\begin{align}
sh_j(x)&=\underline{sh}_j(x)=\limsup_{n\to\infty}-\frac{\log{\mu^{\tilde{\chi}_j}_x(B(\alpha, F^p_n, x, \epsilon))}}{n}\nonumber \\
&=\liminf_{n\to\infty}-\frac{\log{\mu^{\tilde{\chi}_j}_x(B(\alpha, F^p_n, x, \epsilon))}}{n}
=\sum_{i=1}^{k(j)} \gamma_{j,i} \max_{{\bf t}:p({\bf t})\leq 1}\tilde{\chi}_{j,i}({\bf t}). \label{entropy formula-2}
\end{align}
For choosing a large $l$ and a small $\delta$, the equality \eqref{entropy formula-2} holds on a set whose measure is close to $1$. 

We combine the method in \cite{BPS99} 
with the theory of coarse Lyapunov foliations to establish 
the properties of partitions suitable for the Bowen ball 
$B\left(\alpha, F^p_n, x, \epsilon\right)$.

%By the construction in \cite{LY-85-annals-2}, 
%Ledrappier and Young constructed a finite entropy partition $\mathcal{P}$ 
%which satisfies the following properties. 

Below we briefly introduce the method for constructing the partition as presented in \cite{LY-85-annals-2}. 
We use $W^i$ to denote the $i$-th unstable manifold of $f$ from Lemma \ref{choose f}. 
We take a system of $(\epsilon,l)$-local charts $\{\Phi_x\}$ and 
choose $l_0$ such that $\Lambda=\{x\in M\,:\,l(x)\leq l_0\}$ has positive $\mu$-measure, 
and $S=\cup D_\alpha$ as described in \cite[Page 554]{LY-85-annals-2}. Let 
\[
\hat{\xi}_i(x) =
\begin{cases}
W^i(x)\cap D_\alpha & \text{if } x \in D_\alpha, \\
M - S & \text{if } x \notin S.
\end{cases}
\]
Then $\displaystyle \hat{\xi}_i = \bigvee_{n \geq 0} f^n \hat{\xi}_i$, which is an increasing partition 
subordinate to $W^i$. 
We use $\xi^u$ and $\xi^s$ to denote the measurable partition corresponding to 
the unstable manifold and stable manifold of $f$, respectively. 
We define $\mathcal{P}=\xi^u\bigvee \xi^s$. 

Given a small $0<\epsilon<1$, there exist a set $\Gamma\subset M$ with $\mu(\Gamma)>1-\frac{1}{4}\epsilon$, 
an integer $n_0\geq 1$, and a constant $C>1$ such that for every $x\in \Gamma$ and 
any integer $ n \geq n_{0} $, 
\begin{itemize}
  \item[(a)] For all integers $ k, l \geq 1 $ we have
    \begin{align}
      C^{-1} e^{-(l+k) (h+\epsilon)} & \leq \mu\left(\mathcal{P}_{k}^{l}(x)\right) \leq C e^{-(l+k) (h-\epsilon)}, \label{(a)-1} \\
      C^{-1} e^{-l \tilde{h}_j-l \epsilon}         & \leq \mu_{x}^{W^{\tilde{\chi}_j}}\left(\mathcal{P}_{0}^{l}(x)\right) \leq C e^{-l \tilde{h}_j+l \epsilon},\quad\forall\,1\leq j\leq u_0, \label{(a)-2}\\
      C^{-1} e^{-k \tilde{h}_{u_0+j}-k \epsilon}         & \leq \mu_{x}^{W^{\tilde{\chi}_{u_0+j}}}\left(\mathcal{P}_{k}^{0}(x)\right) \leq C e^{-k \tilde{h}_{u_0+j}+k \epsilon}, \quad\forall\,1\leq j\leq s_0, \label{(a)-3}
    \end{align}
    where $h$, $\tilde{h}_j$ and $\tilde{h}_{u+j}$ are the Kolmogorov-Sinai entropies of $f$ 
    with respect to $\mu$, $\mu_{x}^{W^{\tilde{\chi}_j}}$ and $\mu_{x}^{W^{\tilde{\chi}_{u+j}}}$ respectively. 
    \item[(b)]
\begin{align}
 W^{\tilde{\chi}_j}\left(x\right) \cap \bigcap_{n \geq 0} \mathcal{P}_{n}^{0}(x) \supset W^{\tilde{\chi}_j}\left(x, e^{-n_{0}}\right), \quad\forall\,1\leq j\leq u_0,\label{(b)-1}\\ 
 W^{\tilde{\chi}_{u_0+j}}\left(x \right)  \cap \bigcap_{n \geq 0} \mathcal{P}_{0}^{n}(x) \supset W^{\tilde{\chi}_{u_0+j}}\left(x, e^{-n_{0}}\right), \quad\forall\,1\leq j\leq s_0. \label{(b)-2}
\end{align}

  \item[(c)] For each $1\leq i\leq u_0+s_0$, we have 
  \begin{align}
e^{-n(sh_j+\epsilon)}&\leq \mu^{\tilde{\chi}_j}_x(B(\alpha, F^p_n, x, \epsilon))\leq \mu^{\tilde{\chi}_j}_x(B(\alpha, F^p_n, x, 4\epsilon))\leq e^{-n(sh_j-\epsilon)}, \label{(c)-1}\\
e^{-n(\tilde{\delta}_j+\epsilon)}&\leq \mu_x^{W^{\tilde{\chi}_j}}\left( W^{\tilde{\chi}_j}\left(x, e^{-n}\right) \right)\leq e^{-n(\tilde{\delta}_j-\epsilon)}. \label{(c)-2}
\end{align}

\item[(d)] Let $ a$ be the integer part of $\displaystyle 2(1+\epsilon) \max_{ i=1}^L \max_{{\bf t}:p({\bf t})\leq 1}\left(\chi_i({\bf t}) +1\right)   $. 
The property (b) implies that 
\begin{align}
\mathcal{P}_{a n}^{a n}(x) &\subset B(\alpha, F^p_n, x, \epsilon) \subset \mathcal{P}(x), \label{(d)-1}\\ 
\mathcal{P}_{0}^{a n}(x) \cap W^{\tilde{\chi}_j}(x) &\subset B(\alpha, F^p_n, x, \epsilon)\cap W^{\tilde{\chi}_j}(x) \nonumber\\
&\subset \mathcal{P}(x) \cap W^{\tilde{\chi}_j}(x), \quad\forall\,1\leq j\leq u_0, \label{(d)-2}\\
\mathcal{P}_{a n}^{0}(x) \cap W^{\tilde{\chi}_{u_0+j}}(x) &\subset B(\alpha, F^p_n, x, \epsilon)\cap W^{\tilde{\chi}_{u_0+j}}(x)\nonumber \\
&\subset \mathcal{P}(x) \cap W^{\tilde{\chi}_{u_0+j}}(x), \quad\forall\,1\leq j\leq s_0. \label{(d)-3}
\end{align}
\item[(e)] We define
$Q_{n}(x):=\bigcup \mathcal{P}_{a n}^{a n}(y)$, 
where the union is over all $ y \in \Gamma $ for which
\begin{align*}
\mathcal{P}_{0}^{a n}(y) \cap B(\alpha, F^p_n, x, \epsilon) \cap W^{\tilde{\chi}_j}(x) \neq \varnothing ,\quad \forall\,1\leq j\leq u_0, 
\end{align*}
and
\begin{align*}
\mathcal{P}_{a n}^{0}(y) \cap B(\alpha, F^p_n, x, \epsilon) \cap W^{\tilde{\chi}_{u_0+j}}(x) \neq \varnothing, \quad \forall\,1\leq j\leq s_0 .
\end{align*}
By the continuous dependence of 
coarse Lyapunov foliations in the $C^{1+\alpha}$ topology on the base point, we obtain 
\begin{align}\label{P-(16)}
B(\alpha, F^p_n, x, \epsilon) \cap \Gamma \subset Q_{n}(x) \subset B(\alpha, F^p_n, x, 4\epsilon) ,
\end{align}
and for each $ y \in Q_{n}(x) $, we have 
\[
\mathcal{P}_{an}^{an}(y) \subset Q_{n}(x).
\]
\item[(f)] For every $ x \in \Gamma $ and $ n \geq n_{0} $ (increasing $ n_{0} $ if necessary), we have
\begin{align}
W^{\tilde{\chi}_j}(x)\cap B(\alpha, F^p_n, x, \epsilon) \cap \Gamma &\subset Q_{n}(x) \cap \xi^{u}(x) \nonumber\\
&\subset W^{\tilde{\chi}_j}(x)\cap B(\alpha, F^p_n, x, 4\epsilon), \quad \forall\,1\leq j\leq u_0, \label{(f)-1}\\
W^{\tilde{\chi}_{u_0+j}}(x)\cap B(\alpha, F^p_n, x, \epsilon) \cap \Gamma &\subset Q_{n}(x) \cap \xi^{s}(x) \nonumber\\
&\subset W^{\tilde{\chi}_{u_0+j}}(x)\cap B(\alpha, F^p_n, x, 4\epsilon), \quad \forall\,1\leq j\leq s_0.\label{(f)-2}
\end{align}
\end{itemize}

Properties \eqref{(a)-1}-\eqref{(a)-3} follow from the Shannon-McMillan-Bremian theorem 
and its ``leaf-wise'' versions to the partition $\mathcal{P}$. 
The latter is obtained by replacing $W^i(x)$ in $\hat{\xi}_i(x)$ with $W^{\tilde{\chi}_j}$. 
Based on the equality \eqref{entropy formula-2}, 
by choosing $l$ sufficiently large and $\delta$ sufficiently small, 
the properties \eqref{(c)-1} and \eqref{(c)-2} follow. 
Therefore, for sufficiently large $n$, 
by removing a set of measure less than $\frac{1}{4}\epsilon$ with respect to $\mu$, 
we ensure that the remaining set $\Gamma$ satisfies all the above conditions. 
Obviously, we have 
\begin{align*}
h=\sum_{j=1}^{u_0} \tilde{h}_j=\sum_{i=1}^{s_0} \tilde{h}_{u+i}. 
\end{align*}

  Using Lemma \ref{ORH23-Lemma 2.3} and the Borel density lemma (see, for example, \cite[Proposition 3]{BPS99}), 
  there exist an integer $n_1\geq n_0$ and 
a measurable subset $\Gamma_0\subset \Gamma$ with $\mu(\Gamma_0)>1-\frac{1}{2}\epsilon$, 
such that for any $x\in \Gamma_0$, $1\leq j\leq u_0+s_0$, and every $n\geq n_1$, 
\begin{align}
  \mu_x^{W^{\tilde{\chi}_j}}\left( B\left(\alpha, F^p_n, x, \epsilon\right) \cap \Gamma\right) & \geq e^{-n\epsilon}\mu_x^{W^{\tilde{\chi}_j}}\left( B\left(\alpha, F^p_n, x, \epsilon\right)\right), \label{density-(1)} \\
   \mu\left(B\left(\alpha, F^p_n, x, \epsilon\right) \cap \Gamma\right)                             & \geq e^{-n\epsilon} \mu\left(B\left(\alpha, F^p_n, x, \epsilon\right)\right), \label{density-(2)} \\
    \mu_x^{W^{\tilde{\chi}_j}}\left( W^{\tilde{\chi}_j}\left(x, e^{-n}\right)\cap \Gamma \right) &\geq  \frac{1}{2} \mu_x^{W^{\tilde{\chi}_j}}\left( W^{\tilde{\chi}_j}\left(x, e^{-n}\right) \right), \label{density-(3)}
\end{align} 

\begin{proposition}\label{Proposition 4}
  There exists a positive constant $ D=D(\Gamma_0)<1 $ such that for every $ k \geq n_1 $ and $ x \in \Gamma $, we have 
    \begin{align*} 
        \mu_x^{W^{\tilde{\chi}_j}}\left(\mathcal{P}_{k}^{0}(x) \cap \Gamma\right) \geq D, \quad\forall\,1\leq j\leq u_0,\\
        \mu_x^{W^{\tilde{\chi}_{u_0+j}}}\left(\mathcal{P}_{0}^{k}(x) \cap \Gamma\right) \geq D, \quad\forall\,1\leq j\leq s_0.
    \end{align*}
\end{proposition}

\begin{proof}
By \eqref{(b)-1}, for every \( k \geq n_1 \), $1\leq j\leq u_0$, and \( x \in \Gamma \), 
the set \( W^{\tilde{\chi}_j}(x)\cap \mathcal{P}_k^0(x) \cap \Gamma \) contains the set 
\( W^{\tilde{\chi}_j}(x, e^{-n_0}) \cap \Gamma \). It follows from \eqref{(c)-2} and \eqref{density-(1)} that
\[
\mu_x^{W^{\tilde{\chi}_j}}(\mathcal{P}_k^0(x) \cap \Gamma) \geq \mu_x^{W^{\tilde{\chi}_j}}(W^{\tilde{\chi}_j}(x, e^{-n_0})) \geq \frac{1}{2}  e^{-\tilde{\delta}_j n_1 - n_1 \epsilon} .
\]
Similarly, for $1\leq j\leq s_0$, we obtain 
\[
\mu_x^{W^{\tilde{\chi}_{u+j}}}\left(\mathcal{P}_{0}^{k}(x) \cap \Gamma\right)\geq \frac{1}{2}  e^{-\tilde{\delta}_{u+j} n_1 - n_1 \epsilon} . 
\]
Thus, this proposition holds by choosing $\displaystyle D=\min_{1\leq j\leq u_0+s_0} \frac{1}{2}  e^{-\tilde{\delta}_j n_1 - n_1 \epsilon}$. 
\end{proof}

\begin{proposition}\label{Proposition 5}
  For every $ x \in \Gamma $ and $ n \geq n_{1} $, we have
  \begin{align*}
     \mathcal{P}_{an}^{an}(x) \cap W^{\tilde{\chi}_j}(x)=\mathcal{P}_{0}^{an}(x) \cap W^{\tilde{\chi}_j}(x), \quad\forall\,1\leq j\leq u_0, \\
    \mathcal{P}_{an}^{an}(x) \cap W^{\tilde{\chi}_{u_0+j}}(x)=\mathcal{P}_{an}^{0}(x) \cap W^{\tilde{\chi}_{u_0+j}}(x), \quad\forall\,1\leq j\leq s_0. 
  \end{align*}
\end{proposition}
 
\begin{proof}
We will prove the first equality, 
the proof of the second equality follows in a similar manner.

Combining \eqref{(b)-1} with \eqref{(d)-2}, we conclude 
\begin{align*}
\mathcal{P}^{an}_0(x) \cap W^{\tilde{\chi}_j}(x) &\subset \mathcal{P}^{an}_0(x) \cap W^{\tilde{\chi}_j}(x, e^{-n}) \subset \mathcal{P}^{an}_0(x) \cap W^{\tilde{\chi}_j}(x, e^{-n_0}) \\
&\subset \mathcal{P}^{an}_0(x) \cap \mathcal{P}^0_{an}(x) \cap W^{\tilde{\chi}_j}(x) = \mathcal{P}_{an}^{an}(x) \cap W^{\tilde{\chi}_j}(x).
\end{align*}
Since \( \mathcal{P}_{an}^{an}(x) \subset \mathcal{P}_{an}^0(x) \), 
this completes the proof of the first identity. 
\end{proof}

We continue to introduce some modified notions used in \cite{BPS99}.

Fix $ x \in \Gamma_0 $ and an integer $ n \geq n_{1} $. 
  We define two classes $\mathcal{R}(n)$ and $\mathcal{F}(n)$ as follows. 
\begin{align*}
\mathcal{R}(n)&:=\left\{\mathcal{P}_{an}^{an}(y) \subset \mathcal{P}(x): \mathcal{P}_{an}^{an}(y) \cap \Gamma \neq \varnothing\right\}; \\
\mathcal{F}(n)&:=\left\{\mathcal{P}_{an}^{an}(y) \subset \mathcal{P}(x): \mathcal{P}_{an}^{0}(y) \cap \Gamma_0 \neq \varnothing, \text { and } \mathcal{P}_{0}^{an}(y) \cap \Gamma_0 \neq \varnothing\right\}. 
\end{align*}
  We call the elements of these classes  ``rectangles''. 
  
For each set $ A \subset \mathcal{P}(x) $ and every $1\leq j\leq u_0+s_0$, 
we we give the following definitions: 
\begin{align*}
N(n, A)&:=\operatorname{Card}\{R \in \mathcal{R}(n): R \cap A \neq \varnothing\}; \\
N^{j}(n, y, A)&:=\operatorname{Card}\left\{R \in \mathcal{R}(n): R \cap W^{\tilde{\chi}_{j}}(y) \cap \Gamma \cap A \neq \varnothing \right\}; \\
%N^{s}(n, y, A)&:=\operatorname{Card}\left\{R \in \mathcal{R}(n): R \cap W^{\tilde{\chi}_{u+j}}(y) \cap \Gamma \cap A \neq \varnothing,\, \forall\, 1\leq j\leq s\right\}; \\
\widehat{N}^{j}(n, y, A)&:=\operatorname{Card}\left\{R \in \mathcal{F}(n): R \cap W^{\tilde{\chi}_{j}}(y) \cap A \neq \varnothing \right\}.
%\widehat{N}^{s}(n, y, A)&:=\operatorname{Card}\left\{R \in \mathcal{F}(n): R \cap W^{\tilde{\chi}_{u+j}}(y) \cap A \neq \varnothing,\, \forall\, 1\leq j\leq s\right\}.
\end{align*}

\begin{lemma}\label{Lemma 1}
For each $ y \in \mathcal{P}(x) \cap \Gamma $ and $1\leq j\leq u_0+s_0$, 
and integer $ n \geq n_{0} $, we have
\begin{align*}
N^{j}\left(n, y, Q_{n}(y)\right) \leq C\mu_{y}^{W^{\tilde{\chi}_{j}}}\left(B\left(\alpha, F^p_n, x, 4\epsilon\right)\right) \exp(an (\tilde{h}_j+ \epsilon)).
\end{align*}
\end{lemma}

\begin{proof}
We will prove the inequality for the case $1\leq j\leq u_0$; 
the proof of the other case is similar.

Let $ z \in   R \cap \xi^{s}(y) \cap Q_{n}(y) \cap \Gamma $ for some $ R \in \mathcal{R}(n) $. 
Applying Proposition \ref{Proposition 5} and \eqref{(a)-2}, we have 
\begin{align*}
  \mu_{y}^{W^{\tilde{\chi}_{j}}}(R)=\mu_{y}^{W^{\tilde{\chi}_{j}}}\left(\mathcal{P}_{0}^{a n}(z)\right)=\mu_{z}^{W^{\tilde{\chi}_{j}}}\left(\mathcal{P}^{a n}_{0}(z)\right)\geq C^{-1}e^{-an(\tilde{h}_j+\epsilon)}.
\end{align*}
Since the inequality \eqref{(f)-1} and $ R \cap \xi^{u}(y) \cap Q_{n}(y) \neq \varnothing $ 
implies $ R \in \mathcal{R}(n) $, we obtain
\begin{align*}
&\mu_{y}^{W^{\tilde{\chi}_{j}}}\left(B\left(\alpha, F^p_n, x, 4\epsilon\right)\right) \geq \mu_{y}^{W^{\tilde{\chi}_{j}}}\left(Q_{n}(y)\right) \\
\geq& \operatorname{Card}\left\{R \in \mathcal{R}(n): R \cap W^{\tilde{\chi}_{j}}(y) \cap \Gamma \cap A \neq \varnothing\right\} \\
&\quad\quad\quad\quad \cdot \min \left\{\mu_{y}^{W^{\tilde{\chi}_{j}}}(R): R \in \mathcal{R}(n) \text { and } R \cap W^{\tilde{\chi}_{j}}(y) \cap Q_{n}(y) \cap \Gamma \neq \varnothing\right\}\\
\geq & N^{j}\left(n, y, Q_{n}(y)\right) \cdot C^{-1}\exp(-an(\tilde{h}_j+\epsilon)). 
\end{align*} 
\end{proof}

\begin{lemma}\label{Lemma 2}
For each $ y \in \mathcal{P}(x) \cap \Gamma_0 $ and integer $ n \geq n_{1} $, we have
\[
\mu\left(B\left(\alpha, F^p_n, y, \epsilon\right)\right) \leq C  N\left(n, Q_{n}(y)\right) \exp(n\epsilon- 2a n (h- \epsilon)). 
\]
\end{lemma}

\begin{proof}
Mote that $ R \cap Q_{n}(y) \neq \varnothing $ implies $ R \in \mathcal{R}(n)$, 
    by \eqref{(a)-1}, \eqref{P-(16)} and \eqref{density-(2)}, we obtain 
\begin{align*}
\mu\left(B\left(\alpha, F^p_n, y, \epsilon\right)\right)&\leq e^{n\epsilon} \mu\left(B\left(\alpha, F^p_n, y, \epsilon\right)\cap \Gamma\right) \leq e^{n\epsilon} \mu\left(Q_{n}(y) \cap \Gamma\right) \\
&\leq e^{n\epsilon} N\left(n, Q_{n}(y)\right) \cdot \max \left\{\mu(R): R \in \mathcal{R}(n) \text { and } R \cap Q_{n}(y) \neq \varnothing\right\}\\
&\leq  N\left(n, Q_{n}(y)\right) \cdot \exp (n\epsilon- 2a n (h- \epsilon)),
\end{align*}
Thus we finish the proof of this lemma. 
\end{proof}

Next, we estimate the number of elements in the classes $ \mathcal{R}(n) $ and $ \mathcal{F}(n) $. 

\begin{lemma}\label{Lemma 3}
For $ \mu $-almost every $ y \in \mathcal{P}(x) \cap \Gamma_0 $, and $n\geq n_0$, 
there exist an integer $n_2(y)\geq n_1 $ and a constant $C_1$ such that for each $n\geq n_2(y) $, we have 
\begin{align*}
N\left(n+2, Q_{n+2}(y)\right) \leq C_1  \prod_{j=1}^{u_0+s_0}\widehat{N}^{j}\left(n, y, Q_{n}(y)\right) \exp(n\epsilon+4an\epsilon). 
\end{align*}
\end{lemma}

\begin{proof}
By apply $A=\Gamma_0$ to Lemma \ref{ORH23-Lemma 2.3}, for $\mu$-almost every $y\in \Gamma_0$, 
there exists an integer $n_2(y)\geq n_1$ such that for any $n\geq n_2(y)$, we have 
\begin{align*}
\mu(B\left(\alpha, F^p_n, x, \epsilon\right) \cap \Gamma_0)  \geq e^{-n\epsilon} \mu(B\left(\alpha, F^p_n, x, \epsilon\right)). 
\end{align*}
Combining this with \eqref{P-(16)}, for any $ n \geq n_2(y) $, we obtain 
\begin{align}
e^{n\epsilon} \mu\left(Q_{n}(y) \cap \Gamma_0\right) & \geq e^{n\epsilon}  \mu\left(B\left(\alpha, F^p_n, y, \epsilon\right) \cap \Gamma_0\right) \geq \mu\left(B\left(\alpha, F^p_n, y, \epsilon\right)\right)\nonumber \\
& \geq \mu\left(B\left(\alpha, F^p_{n+2}, y, \epsilon\right)\right) \geq \mu\left(Q_{n+2}(y)\right). \label{P-(22)}
\end{align}
For any $m\geq n_2(y)$, by \eqref{(a)-1}, we conclude 
\begin{align}
&\mu\left(Q_{m}(y)\right)=\sum_{\mathcal{P}_{a m}^{a m}(z) \subset Q_{m}(y)} \mu\left(\mathcal{P}_{a m}^{a m}(z)\right) 
\geq N\left(m, Q_{m}(y)\right)  C^{-1} \exp(- 2a m (h+\epsilon));\label{L-1}\\
&\mu\left(Q_{m}(y) \cap \Gamma_0\right)=\sum_{\substack{\mathcal{P}_{am}^{a m}(z) \subset Q_{n}(y)}} \mu\left(\mathcal{P}_{am}^{am}(z) \cap \Gamma_0\right) \leq N_{m} C \exp(-2am  (h-\epsilon)),\label{L-2}
\end{align}
where $ N_{m} $ is the number of rectangles $ \mathcal{P}_{am}^{am}(z) \in \mathcal{R}(m)$ 
with $\mathcal{P}_{am}^{am}(z)\cap \Gamma_0 \neq\varnothing$.
Setting $ m=n+2$ and using \eqref{P-(22)}, \eqref{L-1}, \eqref{L-2}, we obtain
\begin{align}
  N\left(n+2, Q_{n+2}(y)\right) &\leq \mu\left(Q_{n+2}(y)\right) \cdot C\exp(2a (n+2) (h+\epsilon))\nonumber \\
&\leq  \mu\left(Q_{n}(y) \cap \Gamma_0\right) \cdot C\exp(n\epsilon+2a (n+2) (h+\epsilon))\nonumber \\
  &\leq N_{n} \cdot   C^{2}\exp(n\epsilon+4an\epsilon+4a (h+\epsilon))\nonumber\\
&\leq N_{n} \cdot   C_1\exp(n\epsilon+4an\epsilon), \label{P-(23)}
\end{align}
where we choose $C_1=C^2\exp(4a (h+\epsilon))$. For any $ y \in \Gamma$, we have
\begin{align*}
\mathcal{P}_{0}^{a n}(y) \cap \xi^{u}(y) \cap \Gamma_0 \neq \varnothing \text{ and } \mathcal{P}_{a nr}^{0}(y) \cap \xi^{s}(y) \cap \Gamma_0 \neq \varnothing.
\end{align*}

Consider a rectangle $ \mathcal{P}_{an}^{an}(v) \subset Q_{n}(y)$ with 
$\mathcal{P}_{an}^{an}(v)\cap \Gamma_0 \neq \varnothing$. 
The rectangles $ \mathcal{P}_{an}^{0}(v) \cap \mathcal{P}_{0}^{an}(y) $ and 
$ \mathcal{P}_{an}^{0}(y) \cap \mathcal{P}_{0}^{an}(v) $ are in $ \mathcal{F}(n) $ and 
they intersect the coarse unstable foliations and the coarse stable foliations at $ y $ respectively. 
Then, for any $ \mathcal{P}_{an}^{an}(v) \subset Q_{n}(y) $ with 
$ \mathcal{P}_{an}^{an}(v) \cap \Gamma_0 \neq \varnothing$, 
we can associate it with the tuple of sets 
\begin{align*}
&\left( \mathcal{P}_{an}^{0}(v) \cap \mathcal{P}_{0}^{an}(y) \cap W^{\tilde{\chi}_1}(y),\cdots, 
\mathcal{P}_{an}^{0}(v) \cap \mathcal{P}_{0}^{an}(y) \cap W^{\tilde{\chi}_u}(y), \right.\\
&\quad\quad \left.\mathcal{P}_{an}^{0}(y) \cap \mathcal{P}_{0}^{an}(v) \cap W^{\tilde{\chi}_{u+1}}(y), \cdots,
\mathcal{P}_{an}^{0}(y) \cap \mathcal{P}_{0}^{an}(v) \cap W^{\tilde{\chi}_{u+s}}(y) \right) 
\end{align*} 
in
\begin{align*}
  \widehat{N}^{1}\left(n, y, Q_{n}(y)\right)&\times\cdots\times \widehat{N}^{u_0}\left(n, y, Q_{n}(y)\right)\\
  &\times \widehat{N}^{u_0+1}\left(n, y, Q_{n}(y)\right)\times \cdots \times \widehat{N}^{u_0+s_0}\left(n, y, Q_{n}(y)\right)
\end{align*}
This correspondence is injective, 
 since each rectangle in $Q_n(y)$ intersecting $\Gamma_0$
 correspond to a unique tuple of elements in the product space above. 
Thus, we conclude 
\begin{align}\label{L-3}
\prod_{j=1}^{u+s}\widehat{N}^{j}\left(n, y, Q_{n}(y)\right)  \geq N_{n}. 
\end{align}
Combining \eqref{P-(23)} and \eqref{L-3}, we obtain the desired inequality. 
\end{proof}

In the next lemma, we give upper bounds for the number of rectangles in $\mathcal{F}(n)$.
Recall that $m$ is the dimension of the manifold $M$. 

\begin{lemma}\label{Lemma 4}
For each $ x \in \Gamma_0 $, there exists a constant $C_2$ 
such that for any $n\geq n_1$, $1\leq j\leq u_0+s_0$, 
\begin{align*}
   \widehat{N}^{j}(n, x, \mathcal{P}(x)) \leq  C_2 \exp(a n (\tilde{h}_j+2m\epsilon)). 
\end{align*}
\end{lemma}

\begin{proof}
  We only prove the case $ j=1$; 
  the proofs for the remaining cases are similar.

For any $y_i\in \Gamma_0$, 
by Proposition \ref{Proposition 4}, Proposition \ref{Proposition 5} 
and \eqref{(a)-2}, we conclude
\begin{align}
N^{1}\left(n, y_{i}, \mathcal{P}^{0}_{an}\left(y_{i}\right)\right) & \geq \frac{\mu_{y_{i}}^{W^{\tilde{\chi}_{1}}}\left(\mathcal{P}^{0}_{an}\left(y_{i}\right) \cap \Gamma\right)}{\max \left\{\mu_{z}^{W^{\tilde{\chi}_{1}}}\left(\mathcal{P}_{an}^{an}(z)\right): z \in W^{\tilde{\chi}_{u+j}}\left(y_{i}\right) \cap \mathcal{P}(x) \cap \Gamma\right\}} \nonumber\\
& \geq \frac{D}{\max \left\{\mu_{z}^{W^{\tilde{\chi}_{1}}}\left(\mathcal{P}_{an}^{an}(z)\right): z \in W^{\tilde{\chi}_{j}}(y_i) \cap \mathcal{P}(x) \cap \Gamma\right\}} \nonumber\\
& =\frac{D}{\max \left\{\mu_{z}^{W^{\tilde{\chi}_{1}}}\left(\mathcal{P}_{0}^{an}(z)\right): z \in W^{\tilde{\chi}_{j}}(y_i) \cap \mathcal{P}(x) \cap \Gamma\right\}} \nonumber\\
& \geq D C^{-1} \exp\left(a n\left( \tilde{h}_1- \epsilon\right)\right).  \label{P-(25)}
\end{align}
Similarly, for $u_0+1\leq j\leq u_0+s_0$, we have $N^{j}\left(n, y_{i}, \mathcal{P}^{0}_{an}\left(y_{i}\right)\right)\geq D C^{-1} e^{a n( \tilde{h}_j- \epsilon)}$. 
Through \eqref{(a)-1}, we obtain
\begin{align}\label{P-(26)}
N(n, \mathcal{P}(x)) \leq \frac{\mu(\mathcal{P}(x))}{\min \left\{\mu\left(\mathcal{P}_{an}^{an}(z)\right): z \in \mathcal{P}(x) \cap \Gamma\right\}} \leq  C \exp\left(2 a n (h+\epsilon)\right). 
\end{align}

  Fix an integer $t$ with $1\leq t\leq u_0$. 
 Since the partition $ \mathcal{P} $ is countable, 
 there exists a sequence of points $\{ y_{i} \}_i$ such that 
$$
\bigcup_i \left(\mathcal{P}_{0}^{an}\left(y_{i}\right)\cap W^{\tilde{\chi}_{1}}(x) \right)= \mathcal{P}(x)\cap W^{\tilde{\chi}_{1}}(x) , 
$$
and these rectangles are mutually disjoint. 
Without loss of generality, we assume that 
for any $i$ where $ \mathcal{P}^{an}_{0}\left(y_{i}\right) \cap W^{\tilde{\chi}_{1}}(x)\cap \Gamma_0 \neq \varnothing $, 
we have $ y_{i} \in W^{\tilde{\chi}_{1}}(x) \cap \Gamma_0 $. 
For each $i$, there exists a sequence of points $\{y_{i,k}\}_k$ such that 
 $$
\bigcup_k \left(\mathcal{P}_{0}^{an}\left(y_{i,k}\right)\cap W^{\tilde{\chi}_{2}}(x) \right)= \mathcal{P}(x)\cap W^{\tilde{\chi}_{2}}(x) , 
$$
and these rectangles are mutually disjoint. 
Without loss of generality, we assume that 
for any $k$ where $ \mathcal{P}^{an}_{0}\left(y_{i,k}\right) \cap W^{\tilde{\chi}_{2}}(x)\cap \Gamma_0 \neq \varnothing $, 
we have $ y_{i,k} \in W^{\tilde{\chi}_{2}}(x) \cap \Gamma_0 $. 
Continue this process up to the foliation $W^{\tilde{\chi}_{u+s}}$, 
for each $l$, there exists a sequence of points $\{y_{i,k,\cdots,l,r}\}_k$ such that 
 $$
\bigcup_r \left(\mathcal{P}^{0}_{an}\left(y_{i,k,\cdots,l,r}\right)\cap W^{\tilde{\chi}_{2}}(x) \right)= \mathcal{P}(x)\cap W^{\tilde{\chi}_{u+s}}(x) , 
$$
and these rectangles are mutually disjoint. 
We assume that for any $r$ where $$ \mathcal{P}^{an}_{0}\left(y_{i,k,\cdots,l,r}\right) \cap W^{\tilde{\chi}_{u+s}}(x)\cap \Gamma_0 \neq \varnothing, $$
we have $ y_{i,k,\cdots,l,r} \in W^{\tilde{\chi}_{u+s}}(x) \cap \Gamma_0 $. 

  Let $\tilde{I}$ be the set of indices $(i,k,\cdots,l,r)$ for which 
  $\mathcal{P}_{0}^{an}\left(y_{i}\right) \cap \Gamma_0 \neq \varnothing$, then
\begin{align}
&N(n, \mathcal{P}(x))\\
\geq& \sum_{i}\sum_{k}\cdots \sum_{l}\sum_{r} N^{u_0+s_0}\left(n, y_{i,k,\cdots,l,r}, \mathcal{P}^{0}_{an}\left(y_{i,k,\cdots,l,r}\right)\right)\nonumber\\
\geq& \sum_{i}\sum_{k}\cdots \sum_{l}\sum_{r} D C^{-1} \exp\left(a n( \tilde{h}_{u_0+s_0}- \epsilon)\right)\nonumber\\
\geq& \sum_{i}\sum_{k}\cdots \sum_{l}   N^{u_0+s_0-1}\left(n, y_{i,k,\cdots,l}, \mathcal{P}^{0}_{an}\left(y_{i,k,\cdots,l}\right)\right)   D C^{-1} \exp\left(a n( \tilde{h}_{u_0+s_0}- \epsilon)\right)\nonumber\\
\geq& \sum_{i}\sum_{k}\cdots \sum_{l}  D^2 C^{-2} \exp\left(a n\left(\left(\tilde{h}_{u_0+s_0-1}+\tilde{h}_{u_0+s_0}\right)- 2\epsilon\right)\right)\nonumber\\
\geq& \sum_{i} D^{u_0+s_0-1} C^{-(u_0+s_0-1)} \exp\left(a n\left(\left(\sum\limits_{j=2}^{u_0+s_0}\tilde{h}_{j}\right)- (u_0+s_0-1)\epsilon\right)\right)\nonumber\\
=& \sum_{i} D^{u_0+s_0-1} C^{-(u_0+s_0-1)} \exp\left(a n\left(\left(2h-\tilde{h}_1\right)- (u_0+s_0-1) \epsilon\right)\right) \label{P-(24)} 
\end{align}

By the definition of $\hat{N}^{s}(n, x, \mathcal{P}(x))$, for any $1\leq j\leq s$, we have
\begin{align*}
\widehat{N}^{1}(n, x, \mathcal{P}(x))=\operatorname{Card}\left\{i: \mathcal{P}_{0}^{an}\left(y_{i}\right) \cap W^{\tilde{\chi}_{1}}(y_i)\cap\Gamma_0 \neq \varnothing\right\}. 
\end{align*}
Combining \eqref{P-(24)}, \eqref{P-(25)} and \eqref{P-(26)}, we conclude
\begin{align*}
 & C e^{ 2a n (h+\epsilon)}  \geq N(n, \mathcal{P}(x)) \\
 \geq &\hat{N}^{1}(n, x, \mathcal{P}(x)) \cdot D^{u_0+s_0-1} C^{-(u_0+s_0-1)} \exp\left(a n\left(\left(2h-\tilde{h}_1\right)- (u_0+s_0-1) \epsilon\right)\right). 
\end{align*}
Thus the first inequality follows by choosing $C_2=D^{-(u_0+s_0-1)}C^{u_0+s_0}$. 
\end{proof}

Before continuing our proof, we introduce the following notions of slow Bowen entropy.

Let $(X,d)$ be a compact smooth manifold and $\alpha:\mathbb{R}^k\to \text{Diff}(X)$ 
a locally free $\mathbb{R}^k$-action on $X$. 
For $ r>0 $, $ N \in \mathbb{N}$, $s \in \mathbb{R} $ and a nonempty subset $ Z \subset X $, 
we define
\[
  M_{N}^{s}(Z):=\inf \sum_{i \in I} e^{-n_{i} s}
\]
where the infimum is taken over all finite or countable covers $ \left\{B\left(\alpha, F^p_{n_i}, x_{i},\epsilon\right)\right\}_{i \in I} $ of
$ Z$  with $ n_{i} \geq N$ and $x_{i} \in X $ for each $i\in I$. 
Since  $M_{N}^{s}(Z)$ is increasing with respect to $N$, 
the limit $M^{s}(Z)=\lim\limits_{N \rightarrow \infty} M_{N}^{s}(Z)$ exists.
The quantity $ M^{s}(Z) $ exhibits a critical behavior with respect to the parameter $s$, 
transitioning from $\infty$ to $0$ at a certain value. 
We define the {\bf slow Bowen topological entropy} as
\begin{align*}
  sh^{B}(\alpha,Z):=\inf \left\{s: M^{s}(Z)=0\right\}=\sup \left\{s: M^{s}(Z)=\infty\right\}.
\end{align*}
Given a Borel probability measure $\mu$ on $X$ and a point $x\in M$, we define 
\begin{align*}
\underline{sh}_\mu(x):=
\lim_{\epsilon\to 0}\liminf_{s\to \infty}\frac{-\log(\mu(B(\alpha, F^p_s, x, \epsilon)))}{s}. 
\end{align*}

The proof of this lemma is relatively straightforward, so we omit it here; 
readers may refer to the proof of Lemma A.1 in \cite{DQ25} for relevant details. 

\begin{lemma}\label{Billingsley type}
Let $(X,d)$ be a compact smooth manifold and $\alpha:\mathbb{R}^k\to \text{Diff}(X)$ 
a locally free $\mathbb{R}^k$-action on $X$. 
  Let $\mu$ be a Borel probability measure on $X$ and $E$ a Borel subset of $X$ with $\mu(E)>0$. 
  Given $r>0$ and $0<s<\infty$, 
  if $\underline{sh}_{\mu}(x)\geq s$ for all $x\in E$, then $sh^{B}(\alpha,E)\geq s$.
\end{lemma}

The following lemma provides a comparison between the number of rectangles
 in $ \mathcal{F}(n) $ and the number of rectangles in $ \mathcal{R}(n) $.

\begin{lemma}\label{Lemma 5}
For $ \mu $-almost every $ y \in \mathcal{P}(x) \cap \Gamma_0 $, $1\leq j\leq u_0+s_0$, we have
\begin{align*}
\varlimsup\limits_{n \rightarrow+\infty} \dfrac{\widehat{N}^{j}\left(n, y, Q_{n}(y)\right)}{N^{j}\left(n, y, Q_{n}(y)\right)} \exp(-7 a nm \epsilon)<1.
\end{align*}
\end{lemma}

\begin{proof}
We will provide a proof for the inequality in the case $1\leq j\leq u_0$, 
with the proof of the case $u_0+1\leq j\leq u_0+s_0$ following in an analogous manner.

For a fixed $1\leq j\leq u_0$, we define the set $F$ as follows:  
\[
  F:=\left\{y \in \Gamma_0: \varlimsup_{n \rightarrow+\infty} \frac{\hat{N}^{j}\left(n, y, Q_{n}(y)\right)}{N^{j}\left(n, y, Q_{n}(y)\right)} \exp(-7 a n m\epsilon) \geq 1\right\}.
\]
To prove this lemma, it suffices to show that $\mu(F)=0$. 
We proceed by contradiction and assume that $ \mu(F)>0 $. 
Based on \eqref{(f)-1}and \eqref{density-(1)}, 
for each $ n \geq n_{1} $ and $ y \in \Gamma_0 $, we have 
\begin{align}\label{L-4}
\mu_{y}^{W^{\tilde{\chi}_{j}}}\left(Q_{n}(y)\right) \geq \mu_{y}^{W^{\tilde{\chi}_{j}}}\left(B(\alpha, F^p_n, y,\epsilon) \cap \Gamma\right) \geq \exp\left(-n\epsilon\right)\mu_{y}^{W^{\tilde{\chi}_{j}}}\left(B(\alpha, F^p_n, y,\epsilon)\right). 
\end{align}
According to \eqref{(a)-2}, \eqref{(c)-1}, \eqref{L-4} and Proposition \ref{Proposition 5}, we conclude
\begin{align}
N^{j}\left(n, y, Q_{n}(y)\right) & \geq \frac{\mu_{y}^{W^{\tilde{\chi}_{j}}}\left(Q_{n}(y)\right)}{\max \left\{\mu_{z}^{u}\left(\mathcal{P}_{an}^{an}(z)\right): z \in W^{\tilde{\chi}_{j}}(y) \cap \mathcal{P}(x) \cap \Gamma\right\}}\nonumber \\
& \geq e^{-n\epsilon} \frac{\mu_{y}^{W^{\tilde{\chi}_{j}}}\left(B(\alpha, F^p_n, y,\epsilon)\right)}{\max \left\{\mu_{z}^{u}\left(\mathcal{P}_{0}^{an}(z)\right): z \in W^{\tilde{\chi}_{j}}(y) \cap \mathcal{P}(x) \cap \Gamma\right\}} \nonumber \\
& \geq C^{-2}\exp(-n\epsilon- n (sh_j-a\tilde{h}_j+(1+a)\epsilon)). \label{P-(28)}
\end{align}
For each $ y \in F $, by \eqref{P-(28)}, there exists an increasing sequence 
$ \left\{r_{k}\right\}_{k=1}^{\infty}=\left\{r_{k}(y)\right\}_{j=1}^{\infty} $ 
of positive integers such that for any $j$, we have 
\begin{align}
\widehat{N}^{j}\left(m_k, y, Q_{m_k}(y)\right)&\geq \frac{1}{2} N^{j}\left(r_{k}, y, Q_{r_k}(y)\right) \exp(7 a r_k m\epsilon) \nonumber\\
&\geq \frac{1}{2} C^{-2}\exp( -r_k (sh_j-a\tilde{h}_j+(1+a-7am)\epsilon)) \label{P-(29)}
\end{align}
Let $ F^{\prime} \subset F $ be the set of points $ y \in F $ for which the following limit exists: 
\begin{align*}
\lim _{n \rightarrow \infty} -\frac{\log \mu_{y}^{W^{\tilde{\chi}_{j}}}\left(B(\alpha, F^p_n, y,\epsilon)\right)}{n}=sh_j. 
\end{align*}
Then $ \mu\left(F^{\prime}\right)=\mu(F)>0 $, and there exists 
 $ y \in F $ such that 
\begin{align*}
\mu_{y}^{W^{\tilde{\chi}_{j}}}(F)=\mu_{y}^{W^{\tilde{\chi}_{j}}}\left(F^{\prime}\right)=\mu_{y}^{W^{\tilde{\chi}_{j}}}\left(F^{\prime} \cap \xi^{u}(y)\right)>0,\\
\lim _{n \rightarrow \infty} -\frac{\log \mu_{y}^{W^{\tilde{\chi}_{j}}}\left(B(\alpha, F^p_n, x,\epsilon)\right)}{n}=sh_j,
\quad \text{ for any } x\in F^{\prime} \cap W^{\tilde{\chi}_{j}}(y). 
\end{align*}
Let $\mu_{y}^{W^{\tilde{\chi}_{j}}}$ be the Borel probability measure in Lemma \ref{Billingsley type}, 
we obtain
\begin{align}\label{P-(30)}
sh^{B}\left(f,F^{\prime} \cap W^{\tilde{\chi}_{j}}(y)\right)\geq sh_j. 
\end{align}
Let us consider the countable collection of balls
\[
  \mathfrak{B}=\left\{B(\alpha, F^p_{m_j(z)}, z,4\epsilon): z \in F^{\prime} \cap W^{\tilde{\chi}_{j}}(y),\quad j=1,2, \ldots\right\}.
\]
Applying Lemma \ref{ORH23-Lemma 2.2}, for any $ L>0 $, there exists a sequence of points $ \left\{z_{i} \in F^{\prime} \cap \xi^{u}(y)\right\}_{i=1}^{\infty} $ and
a sequence of integers $ \left\{t_{i}\right\}_{i=1}^{\infty} $, where $ t_{i} \in\left\{r_{k}\left(z_{i}\right)\right\}_{k=1}^{\infty} $ and
$t_{i}>L $ for each $ i $ such that we can find a subcover $\mathfrak{C}\subset \mathfrak{B}$ 
of $ F^{\prime} \cap \xi^{u}(y) $,
\[
  \mathfrak{C}=\left\{B(\alpha, F^p_{t_i}, z_i,4\epsilon): i=1,2, \ldots\right\},
\]
and each set $ Q_{t_i}(z_i) $ appears in the sum 
$ \sum\limits_{i: t_{i}=q} \widehat{N}^{u}\left(t_i, z_{i}, Q_{t_i}(z_i)\right) $ 
at most $C_me^{q\epsilon}$ times, then
\begin{align}\label{L-5}
  \sum_{i: t_{i}=q} \hat{N}^{u}\left(q, z_{i}, Q_{t_i}(z_i)\right) \leq C_m \exp\left(q\epsilon\right) \hat{N}^{u}(q, y, \mathcal{P}(y)).
\end{align}
Combining \eqref{P-(29)}, \eqref{L-5} and Lemma \ref{Lemma 4}, we obtain
\begin{align*}
 &M^{sh_j-\epsilon}_{L}\left(f,F^{\prime} \cap \xi^{u}(y)\right)\leq  \sum_{i=1}^{\infty} \exp(-t_{i}\left(sh_j-\epsilon\right)) \\
\leq &\sum_{i=1}^{\infty} \hat{N}^{j}\left(t_{i}, z_{i}, Q_{t_i}(z_i)\right) \cdot 2 C^2 \exp\left(-t_{i}\left(sh_j-\epsilon\right)+ t_i\left(sh_j-a\tilde{h}_j+(1+a-7am)\epsilon\right)\right)   \\
= & 2 C^2 \sum_{q=1}^{\infty} \exp\left(q\left(-a\tilde{h}_j+(2+a-7am)\epsilon\right)\right)  \sum_{i: t_{i}=q} \widehat{N}^{u}\left(q, z_{i}, Q_{t_i}(z_i)\right) \\
\leq& 2 C^2 \sum_{q=1}^{\infty} \exp\left(q\left(-a\tilde{h}_j+(2+a-7am)\epsilon\right)\right)  \cdot C_m \exp\left(q\epsilon\right)\widehat{N}^{u}\left(q, y, \mathcal{P}(y)\right) \\
\leq& 2 C^2 C_m C_2 \sum_{q=1}^{\infty} \exp\left(q\left(3+a-5am\right)\epsilon\right) <\infty. 
\end{align*}
Since $ L $ and $ t_{i} $ can be chosen sufficiently large, we conclude that
\[
sh^{B}\left(f,F^{\prime} \cap \xi^{u}(y)\right) \leq sh_j-\epsilon<sh_j,
\]
which contradicts \eqref{P-(30)}. 
Therefore, we have $ \mu(F)=0 $, proving the second inequality.
\end{proof}

By Lemma \ref{Lemma 5}, for $ \mu $-a.e. $y \in \mathcal{P}(x) \cap \Gamma_0 $, 
there exists an integer $ n_{3}(y) \geq n_{2}(y) $ such that 
for all $ n \geq n_{3}(y) $, $1\leq j\leq u_0+s_0$, we have
\begin{align}
\widehat{N}^{j}\left(n, y, Q_{n}(y)\right)<N^{j}\left(n, y, Q_{n}(y)\right) \exp(7 a nm \epsilon). \label{P-(31)}
\end{align}
By Lusin's theorem, for every $ \epsilon>0 $, 
there exists a compact subset $ \Gamma_{\epsilon} \subset \Gamma_0 $ satisfying
\[
\mu\left(\Gamma_{\epsilon}\right)>\mu(\Gamma_0)-\frac{1}{3}\epsilon>1-\epsilon, \quad \text { and }\quad n_{\epsilon} := \sup \left\{n_{0}, n_{1}(y): y \in \Gamma_{\epsilon}\right\}<\infty,
\]
such that the inequalities \eqref{P-(31)} hold for every $ n \geq n_{\epsilon} $ and $1\leq j\leq u+s$.

\begin{lemma}\label{Lemma 6}
 For every $ \epsilon>0 $, 
 there exists a constant $C_3>0$ such that 
 if $ y \in \Gamma_{\epsilon} $ and $ n \geq n_{\epsilon} $, then 
\begin{align*}
\mu\left(B(\alpha, F^p_{n}, y,4\epsilon)\right)\geq C_3\exp\left(-n\sum_{i=1}^{u_0+s_0}sh_j-20anm^2\epsilon  \right).
\end{align*}
\end{lemma}

\begin{proof}
For any $ z \in \Gamma_{\epsilon} \cap Q_{n}(y) $, $1\leq j\leq u_0+s_0$ and $ n \geq n_{\epsilon} $, 
by \eqref{P-(31)}, we have
\begin{align}
N^{j}\left(n, y, Q_{n}(y)\right) &\leq \widehat{N}^{j}\left(n, y, Q_{n}(y)\right)=\widehat{N}^{j}\left(n, z, Q_{n}(y)\right)\nonumber\\
&<N^{j}\left(n, z, Q_{n}(y)\right) \exp(7 a n m\epsilon),\nonumber\\
N^{j}\left(n, y, Q_{n}(y)\right) &\leq \inf \left\{N^{j}\left(n, z, Q_{n}(y)\right): z \in \Gamma_{\epsilon} \cap Q_{n}(y)\right\} \exp(7 a n m\epsilon). \label{P-(33)}
\end{align}
Since $ N\left(n, Q_{n}(y)\right) $ is equal to the number of rectangles $ R \subset Q_{n}(y) $, 
for $ y \in \Gamma_{\epsilon} $ and $ n \geq n_{\epsilon} $, by using \eqref{P-(28)} and \eqref{P-(33)}, we obtain 
\begin{align}
&N\left(n, Q_{n}(y)\right)\nonumber\\
\geq& \widehat{N}^{1}\left(n, y, Q_{n}(y)\right) \times \prod_{j=2}^{u_0+s_0}\inf \left\{N^{j}\left(n, z, Q_{n}(y)\right): z \in Q_{n}(y)\right\}  \nonumber\\
\geq& N^{1}\left(n, y, Q_{n}(y)\right) \times \prod_{j=2}^{u_0+s_0} N^{j}\left(n, y, Q_{n}(y)\right) \exp\left(- 7an(u_0+s_0-1)m\epsilon \right)\label{P-1}\\
\geq& C^{-2(u_0+s_0)} \exp\left(-n\left(\sum_{i=1}^{u_0+s_0}sh_j-2ah\right)-n(u_0+s_0)(2+a)\epsilon -7an(u_0+s_0-1)m\epsilon   \right). \nonumber
\end{align}
Through \eqref{(a)-1} and \eqref{P-(16)}, we obtain
\begin{align*}
N\left(n, Q_{n}(y)\right) &\leq \frac{\mu\left(Q_{n}(y)\right)}{\min \left\{\mu\left(\mathcal{P}_{an}^{an}(z)\right): z \in Q_{n}(y) \cap \Gamma\right\}} \nonumber\\
&\leq \mu\left(B(\alpha, F^p_{n}, y,4\epsilon)\right) \cdot C \exp( 2a n (h+\epsilon)). %\label{P-4}
\end{align*}
Combining this inequality with \eqref{P-1}, we conclude 
\begin{align*}
&\mu\left(B(\alpha, F^p_{n}, y,4\epsilon)\right)\\
\geq& C^{-2(u_0+s_0)-1}\exp\left(-n\sum_{i=1}^{u_0+s_0}sh_j-n(u_0+s_0)(2+a)\epsilon -7an(u_0+s_0-1)m\epsilon -2an\epsilon  \right)\\
\geq &C^{-2(u_0+s_0)-1}\exp\left(-n\sum_{i=1}^{u_0+s_0}sh_j-20anm^2\epsilon  \right).
\end{align*}
We finish the proof by choosing $C_3=C^{-2(u_0+s_0)-1}$. 
\end{proof}

\begin{lemma}\label{Lemma 7}
  For $ \mu $-a.e. $ y \in \mathcal{P}(x) \cap \Gamma_\epsilon$ and any $n\geq n_\epsilon$, 
  there exists a constant $C_4>0$ such that 
\begin{align*}
   \mu\left(B(\alpha, F^p_{n+2}, y,\epsilon)\right)  \leq C_4  \exp\left(-n \sum_{j=1}^{u_0+s_0}sh_j  \right)\exp\left( 20anm^2\epsilon \right). 
\end{align*}
\end{lemma}

\begin{proof}
By Lemma \ref{Lemma 2} and Lemma \ref{Lemma 3},
for $\mu$-a.e. $ y \in \mathcal{P}(x) \cap \Gamma_\epsilon $ and $ n \geq n_\epsilon $, we have
\begin{align}
       & \mu\left(B(\alpha, F^p_{n+2}, y,\epsilon)\right)\nonumber                                                                                                      \\
  \leq & C N\left(n+2, Q_{n+2}(y)\right)  \exp((n+2)\epsilon- 2a (n+2)(h- \epsilon))  \nonumber  \\
  \leq & CC_1 \prod_{j=1}^{u_0+s_0}\widehat{N}^{j}\left(n, y, Q_{n}(y)\right)  \exp\left( (2n+2)\epsilon-2a(n+2)(h-\epsilon)+4an\epsilon\right)\label{P-5}.
\end{align}
For each $1\leq j\leq u_0+s_0$, by \eqref{P-(31)} and Lemma \ref{Lemma 1}, we obtain
\begin{align}
\widehat{N}^{j}\left(n, y, Q_{n}(y)\right)&<N^{j}\left(n, y, Q_{n}(y)\right) \exp(7 a nm \epsilon)\nonumber\\
&\leq C\mu_{y}^{W^{\tilde{\chi}_{j}}}\left(B\left(\alpha, F^p_n, x, 4\epsilon\right)\right) \exp(an (\tilde{h}_j+\epsilon+7m\epsilon)) \label{P-6}
\end{align}
Combining \eqref{P-5}, \eqref{P-6} and $\displaystyle \sum_{j=1}^{u_0+s_0} \tilde{h}_j=2h$, we derive 
\begin{align*}
    &\mu\left(B(\alpha, F^p_{n+2}, y,\epsilon)\right)  \\
 \leq & C^{u_0+s_0+1}C_1 \prod_{j=1}^{u_0+s_0} \mu_{y}^{W^{\tilde{\chi}_{j}}}\left(B\left(\alpha, F^p_n, x, 4\epsilon\right)\right) \exp\left(an \left(2h+ (u_0+s_0)\epsilon+7m(u_0+s_0)\epsilon\right)\right) \\
&\quad\quad\quad\quad\quad \cdot \exp\left( (2n+2)\epsilon-2a(n+2)(h-\epsilon)+4an\epsilon\right)\\
 \leq & C^{u_0+s_0+1}C_1 \left(\prod_{j=1}^{u_0+s_0} \exp\left( -n(sh_j-\epsilon) \right) \right)\exp\left(an \left(2h+ (u_0+s_0)\epsilon+7m(u_0+s_0)\epsilon\right)\right) \\
&\quad\quad\quad\quad\quad \cdot \exp\left( (2n+2)\epsilon-2a(n+2)(h-\epsilon)+4an\epsilon\right)\\
\leq& C_4  \exp\left(-n \sum_{j=1}^{u_0+s_0}sh_j  \right) \exp\left( (3+6a)n\epsilon+an(u_0+s_0)(1+7m)\epsilon \right)\\
\leq & C_4  \exp\left(-n \sum_{j=1}^{u_0+s_0}sh_j  \right)\exp\left( 20anm^2\epsilon \right),
\end{align*}
where $C_4=C^{u_0+s_0+1}C_1 \exp(2\epsilon-4a(h-\epsilon))$. 
Thus, we complete the proof of this lemma.
\end{proof}

Finally, we provide the proof of Theorem \ref{bk} for the hyperbolic case. 

\begin{proof}[Proof of Theorem \ref{bk} for the hyperbolic case]
Combining Lemma \ref{Lemma 6}, Lemma \ref{Lemma 7} and Lemma \ref{entropy formula}, 
for any $x\in \Gamma_\epsilon$, we have 
\begin{align*}
\sum_{i=1}^{D}\gamma_i \max_{{\bf t}: p({\bf t})\leq 1} \chi_i({\bf t})-20am^2\epsilon&=\sum_{j=1}^{u_0+s_0}sh_j -20am^2\epsilon \\
&\leq \liminf_{n\to \infty}\frac{-\log(\mu(B(\alpha, F^p_n, x, \epsilon)))}{n}\\
&\leq  \limsup_{n\to \infty}\frac{-\log(\mu(B(\alpha, F^p_n, x, \epsilon)))}{n}\\
&\leq \sum_{j=1}^{u_0+s_0}sh_j +20am^2\epsilon=\sum_{i=1}^{D}\gamma_i \max_{{\bf t}: p({\bf t})\leq 1} \chi_i({\bf t})+20am^2\epsilon. 
\end{align*}
Since $\mu(\Gamma_\epsilon)>1-\epsilon$, 
letting $\epsilon\to 0$ completes the proof of Theorem \ref{bk} for the hyperbolic case. 
\end{proof}

\section{Absolutely Continuous Case}\label{set:Absolutely Continuous Case}

In this section, we will give the proof of Theorem \ref{bk} for the absolutely continuous case.

Based on the method in Section \ref{set:Hyperbolic Case}, 
by considering the coarse unstable foliations, 
we can obtain similar results to
Lemma \ref{Lemma 6} and Lemma \ref{Lemma 7}. 
These results imply the following result: 

\begin{proposition}\label{main-unstable}
For $\mu$-a.e. $x\in M$, we have 
\begin{align}
\lim_{\epsilon\to 0}\liminf_{n\to \infty}\frac{-\log(\mu^u_x (B(\alpha, F^p_n, x, \epsilon)))}{n}=\lim_{\epsilon\to 0}\limsup_{n\to \infty}\frac{-\log(\mu^u_x(B(\alpha, F^p_n, x, \epsilon)))}{n},
\end{align}
and this limit is equal to 
$$
\sum_{i=1}^{u}\gamma_i \max_{{\bf t}: p({\bf t})\leq 1} \chi_i({\bf t}).
$$
\end{proposition}

We have a similar result for the stable manifold. 

\begin{proposition}\label{main-stable}
For $\mu$-a.e. $x\in M$, we have 
\begin{align}
\lim_{\epsilon\to 0}\liminf_{n\to \infty}\frac{-\log(\mu^s_x (B(\alpha, F^p_n, x, \epsilon)))}{n}=\lim_{\epsilon\to 0}\limsup_{n\to \infty}\frac{-\log(\mu^s_x(B(\alpha, F^p_n, x, \epsilon)))}{n},
\end{align}
and this limit is equal to 
$$
\sum_{i=u+2}^{L}\gamma_i \max_{{\bf t}: p({\bf t})\leq 1} \chi_i({\bf t}).
$$
\end{proposition}

Now we consider the case where \(\mu\) is absolutely continuous. 
Then \(\gamma_i = d_i=\operatorname{dim}E_i\). Since all volume forms are equivalent, we do not actually require \(\alpha\) to preserve \(\mu\); 
moreover, the limit that we consider remains unchanged by invoking the following well-known result.

\begin{lemma}\label{th1}
Suppose $X$ is a Euclidean space, $\mu$ is a Borel measure on $X$, 
and let $\lambda$ be the Lebesgue measure, 
then $\mu$ is absolutely continuous with respect to $\lambda$ if and only if 
$$
\liminf_{r\to 0}\frac{\mu(B(x,r))}{\lambda(B(x,r))}<\infty,\quad \mu\text{-}a.e.\,\, x\in X.
$$
\end{lemma}

Next, we provide the proof of Theorem \ref{bk} for the absolutely continuous case. 

\begin{proof}[Proof of Theorem \ref{bk} for the absolutely continuous case]
Fix $l>1$ and $\delta>0$. 
For each $x\in\Gamma_{l,\delta}$, we have an embedding $\Phi_x:B(l(x)^{-1})\to M$; then the pullback of $\mu$ restricted to the image of $\Phi_x$ (denoted by $\Phi_x^*\mu(\cdot):=\mu(\Phi_x(\cdot))$) is also absolutely continuous, because $\Phi_x$ is smooth and has bounded derivative. Note that for the Bowen ball, due to Lemma \ref{le1}, it can be controlled on both sides by the images of corresponding rectangles in the tangent space. So we only need to evaluate the limit of these rectangles in the tangent space. By Lemma \ref{th1}, we only need to do so for the standard volume form $\lambda$. 
For $\epsilon>0$ with $\epsilon<<\delta$, 
by direct calculation, we have 
\begin{align*}
\lambda&\bigg(\prod_{i\leq u}B_i\Big(0, \frac{\epsilon e^{-(a_i+2\epsilon)n}}{(m+1)K}\Big)\times B_{u+1}\Big(0, \frac{\epsilon e^{-2\epsilon n}}{(m+1)K}\Big)\times \prod_{i\ge u+2}B_i\Big(0, \frac{\epsilon e^{-(a_i+2\epsilon)n}}{(m+1)K}\Big)\bigg)\\
&\quad\quad\quad=C_1  (K,m,\epsilon)e^{-n(\sum_{i=1}^L d_ia_i)-2m\epsilon n};\\
\lambda&\bigg(\prod_{i\leq u}B_i\Big(0, l(m+1)\epsilon e^{-(a_i-2\epsilon)n}\Big)\times B_{u+1}\Big(0, l(m+1)\epsilon\Big)\times \\
&\quad\quad\quad\prod_{i\ge u+2}B_i\Big(0, l(m+1)\epsilon e^{-(a_i-2\epsilon)n}\Big)\bigg)
=C_2  (l,m,\epsilon)e^{-n(\sum_{i=1}^L d_ia_i)+2m\epsilon n}.\nonumber
\end{align*}
Taking the limit as $n\to\infty$, and letting $\epsilon\to 0$, for $\mu$-a.e., we obtain 
\begin{align*}
&\lim_{\epsilon\to 0}\liminf_{n\to \infty}\frac{-\log(\mu(B(\alpha, F^p_n, x, \epsilon)))}{n}\\
=&\lim_{\epsilon\to 0}\limsup_{n\to \infty}\frac{-\log(\mu(B(\alpha, F^p_n, x, \epsilon)))}{n}=\sum_{i=1}^{L}d_ia_i=\sum_{i=1}^{L} d_i \max_{{\bf t}:p({\bf t})\leq 1}\chi_i({\bf t}). 
\end{align*}
Since $l$ can be arbitrarily large and $\delta$ can be arbitrarily small, the above equalities hold for $\mu$-almost every $x\in M$.
Hence we complete the proof of Theorem \ref{bk} for the absolutely continuous case. 
\end{proof}

Thus, we finish the proof of Theorem \ref{bk}. 

By applying the idea from \cite{Ma81}, we give another proof of the inequality
\begin{align}\label{lower entropy formula}
\lim_{\epsilon\to 0}\liminf_{n\to\infty} -\frac{\log \mu\left(B(\alpha, F^p_n, x,\epsilon)\right)}{n} \geq \sum_{i=1}^{L} \gamma_i \max_{{\bf t}:p({\bf t})\leq 1}\chi_i({\bf t}). 
\end{align}
It is well known that $M$ can be smoothly embedded into $\mathbb{R}^{2m+1}$. 
We denote the embedding map by $\iota$. 
Thus, $\iota(M)$ is a smooth submanifold of $\mathbb{R}^{2m+1}$, 
we then pick a bounded tubular neighborhood $N$ of $\iota(M)$, 
which we can regard as a normal bundle of $\iota(M)$. 
For any $f\in \text{Diff}^{1+r}(M)$ preserving $\mu$, 
we can define $F\in \text{Diff}^{1+r}(N)$ such that $F\circ \iota=\iota\circ f$; 
furthermore, $\iota(M)$ is a closed invariant set of $F$, $F|\iota(M)$ preserve $\iota_*\mu$. 
Then the dynamics of $f$ on $M$ is the same (in the smooth sense) as the dynamics of 
$F|\iota(M)$ on $\iota(M)$. The idea to define $F$ is through local charts, 
and let $F$ preserve (as $f$) the base $\iota(M)$, but contract in all the normal directions. 
In this way, we can identify $\iota(M)$ with $M$, 
and without confusion still use the same notations, for example, 
$\alpha$ is action, $d(\cdot,\cdot)$ is the metric.
 Below, by Lemma \ref{th1}, we can always use $\mu$ as a volume form on $M$, 
 or induced volume form on any submanifolds of $M$.

\begin{definition}
$E$ is a normed space with the splitting $E=E_1\oplus E_2$. We call a subset $G\subset E$ is a 
$(E_1,E_2)$-graph if there exists an open $U\subset E_2$ and 
a $C^1$ map $\Psi:U\to E_1$ satisfying $G=\{x+\Psi(x)|x\in U\}$. 
The dispersion of $G$ is the number 
$$
\sup\{\|\Psi(x)-\Psi(y)\|/\|x-y\|,\,\;\forall\, x,y\in U\}.
$$
\end{definition}

For the specific $f$, we have splitting $TM=E^u\oplus E^{cs}$, 
where $E^{cs}:=E^c\oplus E^s$. Fix $\varepsilon>0$, by Egorov's theorem, 
we can choose a compact set $L_\epsilon\subset M$ with $\mu(L_\epsilon)\ge 1-\varepsilon$ such that 
the splitting is continuous with the change of $x\in L_\epsilon$. 
We also need to require $L_\epsilon$ to meet that all holonomy maps from unstable manifold 
to unstable manifold in the local charts are continuous with respect to the base points, 
and $L_\epsilon \subset \Gamma_{l,\delta}$ for sufficiently large $l$ and sufficiently small $\delta$. 

\begin{lemma}\label{le6}
For sufficiently small $w>0$, 
there exists $v>0$ such that for all $ x\in L_\epsilon,\;\mu\text{-}a.e. \;y$, 
and $d(x,y)<v$, the set $y+E^u(x)$ is a $(E^u(x),E^{cs}(x))$-graph with dispersion $\leq w$. 
Moreover, if $y\in E^{cs}(x)$, $\epsilon>0$ is small with $\epsilon<l(y)^{-1}$, 
and $n$ is sufficiently large, then  
$$
\mu_x^u((y+E^u(x))\cap B(\alpha,F^p_n,x,\epsilon))< \mu_x^u(E^u(x)\cap B(\alpha,F^p_n,x,3\epsilon)).
$$ 
Here, we use the same scale of volume form on $E^u$.
\end{lemma}

\begin{proof}
The first assertion is straightforward, 
so we only need to prove the second result. To this end, note that 
$$
(y+E^u(x))\cap B(\alpha,F^p_n,x,\epsilon)\subset (y+E^u(x))\cap B(\alpha,F^p_n,y,2\epsilon),
$$ 
since $y+E^{u}(x)$ is a $(E^u(x),E^{cs}(x))$-graph with sufficiently small dispersion $c$. 
Translating these sets back to $x$ then yields the desired inequality.
\end{proof}

Based on Proposition \ref{main-unstable}, we can easily conclude the following lemma:
\begin{lemma}\label{le7}
For $\mu$-a.e. $x\in M$,
\begin{align*}
\lim_{\epsilon\to 0}\liminf_{n\to \infty}\frac{-\log\mu_x^u(E^u(x)\cap B(\alpha, F^p_n, x, \epsilon))}{n}=\lim_{\epsilon\to 0}\limsup_{n\to \infty}\frac{-\log \mu_x^u(E^u(x)\cap B(\alpha, F^p_n, x, \epsilon))}{n},
\end{align*}
and equals to $$\sum_{i=1}^{u}d_i \max_{{\bf t}: p({\bf t})\leq 1} \chi_i({\bf t}).$$
\end{lemma}

A similar result holds for the center-stable case.

\begin{proposition}\label{propp1}
For $\mu$-a.e. $x\in M$,
\begin{align*}
\lim_{\epsilon\to 0}\liminf_{n\to \infty}\frac{-\log\mu^{cs}_x(E^{cs}(x)\cap B(\alpha, F^p_n, x, \epsilon))}{n}=\lim_{\epsilon\to 0}\limsup_{n\to \infty}\frac{-\log \mu_x^{cs}(E^{cs}(x)\cap B(\alpha, F^p_n, x, \epsilon))}{n},
\end{align*}
and equals to 
$$
\sum_{i=u+2}^{L}d_i \max_{{\bf t}: p({\bf t})\leq 1} \chi_i({\bf t}).
$$
\end{proposition}

\begin{proof}
Note that the Bowen ball in the $E^c$ direction will not expand more than $\epsilon$ or contract more than subexponentially $e^{-n\epsilon}$; hence we can use integration on $E^c$ and Lemma \ref{main-stable}, then get the result.
\end{proof}

Now we give the proof of inequality \eqref{lower entropy formula}. 
For any $x\in L$, there exists $C_0>0$ such that for sufficiently large $n$, 
$$
\mu(B(\alpha, F^p_n,x,\epsilon))=C_0\int_{E^{cs}(x)}\mu_x^u((y+E^u(x))\cap B(\alpha,F^p_n,x,\epsilon)) \rd\tilde{\mu}(y), 
$$ 
where $\tilde{\mu}$ is the factor-measure on $ B(x,\epsilon)/\xi^u$ defined by 
$\displaystyle \tilde{\mu}(E)=\mu\left(\bigcup_{\xi^u(x)\in E} \xi^u(x)\right)$. 
Based on Lemma \ref{le6}, there exists a constant $C_1>0$ such that 
 $$\mu(B(\alpha, F^p_n,x,\epsilon))\leq C_1\mu^{cs}_x(E^{cs}(x)\cap B(\alpha,F^p_n,x,\epsilon))\mu^u_x((E^u(x))\cap B(\alpha,F^p_n,x,3\epsilon)). 
$$ 
Taking the logarithm of both sides and applying Lemma \ref{le7} and Proposition \ref{propp1}, 
we obtain inequality \eqref{lower entropy formula}.

\section{Open Questions and Possible Characterization}\label{set:Open Questions}

We point out some open questions and possible characterizations of slow entropy for further study. 
Of course, the first one is what we left:

\begin{question}
Prove the slow entropy formula for a general invariant measure.
\end{question}

It is not so easy to deal with this problem, as we already mentioned a little bit in the introduction. However, the next one may be a little more interesting.

\begin{question}
Generalize the definition of slow entropy to more general group actions, for example, 
free or amenable group actions, and prove an entropy formula for these actions.
\end{question}

Slow entropy is used to determine whether an action has a smooth realization \cite{KT97}, 
so one may encounter problems of general group actions. 
Then this question is meaningful. However, there are some difficulties for this, 
for instance, Theorem \ref{THD} no longer holds. 

Next, we point out a direction to generalize a slow entropy version of SMB theorem. 
In contrast to the main result in \cite{OW83} and \cite{Lin01}, it seems not possible to have an analogy of SMB Theorem for slow entropy. But, when restricting to one Weyl Chamber or an open cone, the following question arises:

\begin{question}
Can one obtain a SMB-type theorem for slow entropy 
when restricting the action to a certain Weyl Chamber or open cone? 
Specifically, consider a $\mathbb{Z}^k$ action $\alpha$ preserving an ergodic measure $\mu$, 
and pick one Weyl Chamber $\mathcal{C}$, let $\xi$ be a measurable partition on $M$. 
Define $$\xi^{\alpha,\mathcal{C}}_n:=\bigvee_{p({\bf t})\leq n,\,t\in\mathbb{Z}^k\cap \mathcal{C}}\alpha(-{\bf t})\xi.$$
For $\mu$ a.e. $x$, 
$$
\lim_{n\to\infty}\frac{-\log{\mu(\xi^{\alpha,\mathcal{C}}_n(x))}}{n}=\sum_{i}\gamma_i\max_{p({\bf t})\leq 1,{\bf t}\in \mathcal{C}} \chi_i^+({\bf t}).
$$
\end{question}

Conjecturally, it is possible to get such a type of result on any open convex cone, 
and it will have the form of the slow entropy formula but with one side. 
This question may be useful to answer the first question.

There are many characterizations for metric entropy. 
We wish to make some good analogies to those. 
Here, we only consider an extension via Poincar\'e recurrence, see, e.g., \cite{Var09}. 
Similarly, define $$R_n(x,\epsilon)=\inf\{p({\bf t})\,:\,\alpha({\bf t})x\in B(\alpha,F^p_n,x,\epsilon),\;{\bf t}\in \mathbb{Z}^k\}.$$

\begin{question}
Suppose $\alpha$ is a free abelian action, $\mu$ is ergodic, do the following two limits
$$
\lim_{\epsilon\to 0}\liminf_{n\to\infty}\frac{\log(R_n(x,\epsilon))}{n}\;\text{and}\;\lim_{\epsilon\to 0}\limsup_{n\to\infty}\frac{\log(R_n(x,\epsilon))}{n}
$$
exist and coincide for $\mu$-a.e. $x$? 
Furthermore, the limit is equal to 
$$
c(p)\sum_{i}\gamma_i \max_{{\bf t}: p({\bf t})\leq 1} \chi_i({\bf t}),
$$ 
here $c(p)$ is a constant depend only on the norm $p$, 
which may have form $vol(p)^{-\frac{1}{k}}$.
\end{question}

\appendices

\section*{Appendix}
\section{Equivalent definitions of slow entropy for smooth abelian actions}

In this appendix, we show that for smooth abelian actions, slow entropy type invariant defined through Hamming metric coincides with that defined through Bowen balls.

\begin{theorem}\label{m}
Let $\alpha$ be a $C^{1+r}$ $\mathbb Z^k$ or $\mathbb R^k$ action on a manifold $M$, preserving an ergodic probability measure $\mu$. Endow on $\mathbb R^k$ with a norm $p$, on $M$ a metric $d$. Then 
$$
sh_\mu(\alpha,p)=sh_\mu^H(\alpha,p).
$$
\end{theorem}

Notice that $sh_\mu^H(\alpha,p)$  is by definition an invariant of measurable isomorphism. However, $sh_\mu(\alpha,p)$ {\it a priori} is not. 

\begin{corollary} The Bowen entropy $sh_\mu(\alpha,p)$ is invariant under measurable isomorphism.
\end{corollary}

By using the standard suspension construction, it suffices to deal with 
the case of $\mathbb R^k$ actions. Note that, by Proposition 2 in \cite{KT97}, 
we conclude that the Hamming entropy is no greater than Bowen entropy. So we need to prove the other one under the smoothness assumption. 
Let us remark that we need the $C^{1+r}$ assumption because we are going to use Lemma \ref{le1}, which is induced by Pesin theory.

Let's also remark that it is hopeful to obtain similar results under the condition that the entropy function is of shape polytope, or under expansiveness.

%In this section, we give the proof of our main result.

Given a small $\epsilon$, we choose a finite measurable partition $\xi$ such that 
$\displaystyle \mu\left(\bigcup_{P\in\xi} \partial P\right)=0$ and 
$\operatorname{diam}(P)< \epsilon$ for any $P\in \xi$. 
Denote $U_{\beta(\delta)}$ as an open $\beta(\delta)$-neighborhood of $\displaystyle \bigcup_{P\in\xi} \partial P$ with $\mu(U_{\beta(\delta)})\le\delta^2$. 
Let $F(n)$ be the Euclidean ball centered at origin of radius $n$ defined by norm $p$, which forms a F\o lner set of $\mathbb{R}^k$.
Let $f$ be the characteristic function of the set $U_{\beta(\delta)}$, 
we denote by $M_\delta$ the set of points $x$ such that for all $n\ge N(\delta)$, 
$$
\frac{1}{|F(n)|}\int_{ F(n)}f(\alpha({\bf k})x)\,\rd {\bf k}\le \delta. 
$$ 
By Chebyshev's inequality, we obtain $\mu(M_\delta)>1-\delta$.

Here, we consider the negative or positive proportional exponents as the same exponent. This induces an equivalence class in Lyapunov exponent functionals. 
Let $\chi_i \,(1\le i\le r)$ be the Lyapunov exponents in the equivalent class corresponding to the Lyapunov decomposition of $\alpha$. 
Choose $L_i$ to be the line direction that maximizes the $\chi_i$, i.e., maximizes $\frac{|\chi_i({\bf t})|}{p({\bf t})}$. Then divide the $\mathbb R^k$ 
by closed cones ${\mathcal C}_i$ containing $L_i$. 
Note that, since the Lyapunov exponent functionals are linear, 
it is possible to choose these $L_i$s and thus have a division such that the Lebesgue measure of the intersection of each cone with the ball centered at the origin of radius $R$ in norm $p$ is at least $c$ times the Lebesgue measure of the ball for any $c$. We refer to each such cone as a $c$-cone.

The following estimate is a part of the proof in the rank one case, see (1.3) in \cite{Ka80}.

\begin{lemma}\label{l1}
Given $\epsilon>0$ small, and a code $\omega$ of length $n$, then the number of the codes that have Hamming distance to $\omega$ less than $\epsilon$ is $O(e^{(-2n\epsilon\log\epsilon)})$. What's more, similar result also holds in continuous cases.
\end{lemma}

Consider a minimal cover $\mathcal K_n$ of a subset $K_{\epsilon,n}\subset M_\delta$ by Hamming balls $B_H(x, F^p_{n},c\epsilon^{2k})$ with $\mu(K_{\epsilon,n})\ge 1-2\delta$. We are going to prove that for any $x\in M_\delta$, the Hamming ball $B_H(x, F^p_{n},c\epsilon^{2k})$ can be covered by certain amount of Bowen balls by disregarding a subset of very small conditional measure. This is essential in the subsequent argument. In fact, we prove that for most points in a Hamming ball, their codes will be close in certain directions, thus most of their iterations are in the same atom of the partition $\xi$.

\begin{lemma}\label{l2}
For any $\epsilon\ll \frac{1}{r^2k^4}$, 
there exists a small $\ell>0$ such that for any $x\in K_{\epsilon,n}\subset M_{\delta}$, 
there exists a subset $X\subset B_H(x, F^p_{n},c\epsilon^{2k})$ 
with $\mu(X)\ge (1-\ell)\mu(B_H(x, F^p_{n},c\epsilon^{2k}))$ where, 
for any $y\in X$, the Hamming distance between $y$ and $x$ 
along any direction $L_i$ is at most $\sqrt{\epsilon}$. 
\end{lemma}

\begin{proof}

We abbreviate $B_H(x, F^p_{n},c\epsilon^{2k})$ as $B_H$ (the Hamming ball) for short. %Denote ${\mathcal C}_i$  the cone  in that contains $L_i$.

Pick an $\epsilon$-cone $\mathcal C\subset \mathcal C_i$ around $L_i$; 
the space of all the rays inside the $\mathcal C\cap F^p_{n}$ can be canonically identified with 
a $k-1$ dimensional disc $\mathcal D$ in the unit $k$-sphere. 
Let $m$ be the conditional measure of Lebesgue measure restricted on $\mathcal D$. 
Then we have that $m(\mathcal D)= \tau \epsilon^{k-1}$, 
where $\tau$ is a constant depending only on the dimension. 
Below, let $\mathbf a\in\mathcal D$ denote a ray segment in $\mathcal C\cap F^p_{n}$.

Define a function $\Phi$ on the product space $B_H\times \mathcal D$ by 
$$
\Phi(y,\mathbf{a}):=d^H_{\mathbf a}(x,y),\; \text{for}\; (y,\mathbf a)\in B_H\times \mathcal D.
$$ 
Here, $d^H_{\mathbf a}(x,y)$ is defined to be the Hamming distance of points $x$ and $y$ along the ray $\mathbf a$.

We claim that there exists an $\mathbf a\in \mathcal D$ such that 
there exists a subset $X_i\subset B_H$ with $\mu(X_i)\ge (1-\frac{1}{2rk})\mu (B_H)$ where, 
for any $y\in X_i$, $\Phi(y,\mathbf a)\le \sqrt{\epsilon}$. We prove it by contradiction.

Assume that for any $\mathbf a$, 
$\mu(\{y:\Phi(y,\mathbf a)>\sqrt{\epsilon}\})\ge \frac{1}{2rk}\mu(B_H)$. We want to estimate the integral 
$$
\int_{B_H\times \mathcal D}\Phi(y,\mathbf a)\,\rd(\mu\times m).
$$ 
On one hand, we have
\begin{equation}\label{e1-2}
\int_{B_H\times \mathcal D}\Phi(y,\mathbf a)\,\rd (\mu\times m)\ge \frac{1}{2rk}\mu(B_H)\sqrt{\epsilon}\tau\epsilon^{k-1}=\frac{\tau}{2rk}\mu(B_H)\epsilon^{k-\frac{1}{2}}.
\end{equation}
On the other hand, for fixed $y\in B_H$, as $d^H_{F^p_{n}}(x,y)\le c\epsilon^{2k}$, 
then by the choice of $c$, 
$$
|F(n)\cap \mathcal C_i|\ge c|F(n)|,
$$ 
it follows that $d^H_{{\mathcal C}_i}(x,y)\le \epsilon^{2k}$. 
Thus by the Chebychev's inequality 
$$
m\left(\{\mathbf a\in\mathcal D:\Phi(y,\mathbf a)>\epsilon\}\right)\le \frac{\epsilon^{2k}}{\tau\epsilon^{k-1}\epsilon}=\frac{\epsilon^{k}}{\tau}.
$$
Since $\Phi(y,\mathbf a)\leq 1$, we have 
\begin{align}
& \int_{B_H\times \mathcal D}\Phi(y,\mathbf a)\,\rd (\mu\times m)\nonumber\\
\leq & \int_{B_H\times \mathcal (D\cap\{\mathbf a\in\mathcal D:\Phi(y,\mathbf a)>\epsilon\} )}\Phi(y,\mathbf a)\,\rd (\mu\times m)+\int_{B_H\times \mathcal (D\backslash \{\mathbf a\in\mathcal D:\Phi(y,\mathbf a)>\epsilon\} )}\Phi(y,\mathbf a)\,\rd (\mu\times m)
\nonumber\\
\leq & \mu(B_H)\frac{\epsilon^{k}}{\tau}\cdot 1+\mu(B_H)(m(\mathcal D)-\frac{\epsilon^{k}}{\tau})\cdot \epsilon\nonumber\\
\leq & \mu(B_H)\left(\frac{1}{\tau}+\tau\right)\epsilon^{k}.\label{e2-2}
\end{align}
Combining (\ref{e1-2}) with (\ref{e2-2}), 
we obtain $\epsilon\geq \left(\frac{\tau^2}{2rk(1+\tau^2)}\right)^2$, 
which contradicts the smallness of $\epsilon$ chosen at the beginning.

Applying the above argument to each Lyapunov exponent $\chi_i$, 
we obtain $r$ subsets $X_i\subset B_H$ for $1\le i\le r$. 
Let $\displaystyle X=\bigcap_{i=1}^r X_i$, then we obtain $\mu(X)\ge (1-\frac{1}{k})\mu(B_H)$, and for any $y\in X$, 
we have $d^H_{\cup \mathbf a_i}(x,y)\le \sqrt{\epsilon}$.
\end{proof}

Lemma \ref{le1} tells us that Bowen ball typically is determined by certain directions 
corresponding to the coarse Lyapunov exponents. 

\begin{lemma}\label{l3}
With the same notations as in Lemma \ref{l2}, there exist constants $K,C>0$ such that 
the set $X$ can be covered by at most $Ke^{(-Cn\sqrt{\epsilon}\log\epsilon)}$ Bowen balls of the form $B_d(y,F^p_{n},\sqrt{\epsilon})$ for sufficiently large $n$.
\end{lemma}

\begin{proof}
Without loss of generality, we can assume $X\subset B_H(x)\cap B_d(x,\epsilon)$. 

By Lemma \ref{le1}, a Bowen ball is typically determined by certain directions corresponding 
to the coarse Lyapunov exponents. For these directions, we are flexible to 
allow an $\epsilon$ fluctuation, which will at most give a multiplicative error term $K_1e^{C_1n\epsilon}$ with some constants $K_1, C_1>0$ for large $n$.

In the case of zero Lyapunov exponents, along these directions, 
the Bowen ball expands at most $K_1e^{C_1n\epsilon}$ for sufficiently large $n$, 
and hence this also generates a multiplicative error term $K_1e^{C_1n\epsilon}$.

Combining Lemma \ref{l1}, Lemma \ref{l2} and the fact that each atom of the partition $\xi$ has diameter less than $\epsilon$, we conclude that for sufficiently large $n$, 
$X$ can be covered by at most $Ke^{(-Cn\sqrt{\epsilon}\log\epsilon)}$ Bowen balls $B_d(y,F^p_{n},\sqrt{\epsilon})$, where $K=K_1^2$ and $C=C_1^2$.
\end{proof}

Based on the argument from Lemma \ref{l3}, we conclude the following result. 

\begin{proposition}\label{a-p1}
There exists a constant $c$ determined by the angles between the Lyapunov planes,  
a constant $C$ that depends on $r$, 
and a small constant $\ell$ which appears in Lemma \ref{l2} and depends on $\epsilon,k$ such that 
$$
S_d(\alpha, F^p_{n}, \sqrt{\epsilon},\delta)\le Ke^{(-Cn\sqrt{\epsilon}\log\epsilon)}S_\xi^H(\alpha,F^p_{n},c\epsilon^{2k},(1-\ell)(1-\delta)+\delta).
$$
\end{proposition}

We are now ready to prove Theorem \ref{m}. 

\begin{proof}[Proof of Theorem \ref{m}]
The Theorem \ref{m} follows immediately by combining Proposition \ref{a-p1} 
above with Proposition 2 in \cite{KT97}, which gives the estimation in the opposite direction. Thus, we finish the proof of this theorem. 
\end{proof}

\end{document}